\documentclass[11pt,reqno]{article}
\usepackage{amsgen, amsmath, amsfonts, amsthm, amssymb, enumerate,amscd,graphics,dsfont,comment,tikz-cd}

\newtheorem{defn}{Definition}[section]
\newtheorem{prop}[defn]{Proposition}
\newtheorem{lem}[defn]{Lemma}
\newtheorem{thm}[defn]{Theorem}

\newtheorem{rem}[defn]{Remark}
\newtheorem {conj}[defn]{Conjecture}

\newcommand {\ZZ}{{\mathds Z}}

\newcommand {\XX}{{\mathcal X}}

\newcommand {\C}{{\mathds C}}

\newcommand {\CC}{{\mathcal C}}
\newcommand {\Q}{{\mathds Q}}
\newcommand {\QQ}{{\bf Q}}

\newcommand {\OO}{{\mathcal O}}

\newcommand {\HH}{{\mathds H}}

\newcommand {\LL}{{\mathcal L}}

\newcommand {\M}{{\mathcal M}}

\newcommand {\CP}{{\mathds P}}

\newcommand {\GG}{{\mathcal G}}
\newcommand {\AB}{{\mathds A}}
\newcommand {\TL}{\tilde{L}}

\newcommand {\TK}{{\tilde{K}}}

\newcommand{\w}{{\omega}}
\newcommand{\TC}{\tilde{C}}

\newcommand{\BC}{\bar{C}}

\def\Ker{\operatorname{Ker}}

\def\div{\operatorname{div}}
\def\deg{\operatorname{deg}}
\def\mod{\operatorname{mod}}

\def\Im{\operatorname{Im}}
\def\ord{\operatorname{ord}}
\def\reg{\operatorname{reg}}

\title{Algebraic cycles and values of Green's functions}
\author{Ramesh Sreekantan}

\begin{document}
\baselineskip=17pt

\maketitle

\begin{abstract}
	We construct indecomposable cycles in the motivic cohomology group $H^3_{\M}(A,\Q(2))$ where $A$ is an Abelian surface over a number field or the function field of a base.  When $A$ is the self product of the universal  elliptic curve over a modular curve, these cycles can be used to prove  algebraicity results for values of higher Green's functions, similar to a conjecture of Gross, Kohnen and Zagier. We formulate a conjecture which relates our work with the recent work of Bruinier-Ehlen-Yang on the conjecture of Gross-Kohnen-Zagier. 
	
\end{abstract}

\tableofcontents

\section{Introduction}

This paper has two themes. The first is the construction of `interesting' cycles in the motivic cohomology of Abelian surfaces - specifically the generic fibre of the universal family over a Siegel modular threefold. The second is the relation of those cycles with the algebraicity of values of Green's functions. This allows us to formulate a conjectural relationship between weakly holomorphic modular forms, on the one hand and motivic cycles, on the other and suggests a motivic interpretation of some Borcherd's lifts. 

\subsection{Motivic cycles.}

Let $A$ be an Abelian surface over a field $K$. The cycles we construct lie in the motivic cohomology group $H^3_{\M}(A,\Q(2))$. This  is the same as the higher Chow group $CH^2(A,1)$ and the group $H^1(A,{\mathcal K}_2)$. In this group there are certain cycles which are `decomposable' -- namely coming from other motivic cohomology  groups using the product structure -- but the cycles we construct are indecomposable. 

We have the following theorem (Theorem \ref{motiviccycle}): 

\begin{thm}

Let $\Delta$ be an integer which belongs to a certain infinite family of discriminants of real quadratic fields. There exists an indecomposable motivic cycle $\Xi^c_{\Delta}$ in $H^3_{\M}(A_{\eta},\Q(2))$ where $A_{\eta}$ is the generic abelian surface over the Siegel modular threefold with full level $2$ structure. This cycle is defined in the fibres over the complement of a component $H^c_{\Delta}$ of the Humbert surface $H_{\Delta}$ and has a boundary supported  precisely on that component. 
	
\end{thm}

To construct the cycle we use generalizations of two classical results. To a principally polarised Abelian surface one can canonically associate a configuration of six lines tangent to a conic in $\CP^2$. These six lines will intersect in $15$ points. A classical result of Humbert's \cite{humb} gives a beautiful criteria for an Abelian surface to have multiplication by $\Q(\sqrt{5})$ in terms of this configuration in $\CP^2$. He showed that multiplication by $\Q({\sqrt 5})$ holds if and only if  there exists a second conic passing through $5$ of the $15$ points and is further tangent to the sixth line. Birkenhake and Wilhelm \cite{biwi} generalized this to a large class of real quadratic fields, the fields of disciminant $\Delta$ in the theorem, and showed that when an Abelian surface has multiplication by $\sqrt{\Delta}$ there is an exceptional {\em rational} curve in the corresponding $\CP^2$ passing through some of the $15$ points and meeting the remaining lines at points of even multiplicity. 

The second classical result we use is the fact that there exists a conic passing through any $5$ points in $\CP^2$ in general position. The generalization of this is that there always exists a rational curve of degree $d$ passing though $3d-1$ points. We use this theorem to show that their exists a curve in the generic fibre that deforms the exceptional curve and use that  to construct a  motivic cycle. The exact number of such curves is the celebrated theorem of Kontsevich-Manin \cite{koma} and Ruan-Tian \cite{ruti}. 

\subsection{Algebraicity of values of higher Green's functions.}

The second theme of this paper is algebraicity of values of Green's functions. Our motivic cycles  lie in the fibres of the complement of a divisor on the Siegel modular threefold. We can further restrict them to families over modular or Shimura curves to get cycles defined in the fibres of the complement of some points - which turn out to be CM points.

It turns out that the regulator of  motivic cycles in $H^3_{\M}(A_{\eta},\Q(2))$, when $A_{\eta}$ is the generic fibre of the  universal Abelian surface over a modular or Shimura curve, can be expressed in terms of certain higher Green's functions. These are  functions introduced by Gross and Zagier in the modular curve case. Precisely, a theorem of Zhang \cite{zhan} implies that  the matrix coefficients of the Green's current of CM cycles can be expressed in terms of  Green's functions of degree $2$. The boundary of our cycle is a linear combination of CM cycles and hence the regulator can be expressed a sum of the Green's currents of the cycles appearing in the boundary. More generally, the regulator of a cycle in $H^{2j+1}_{\M}(A^{j}_{\eta},\Q(j+1))$ can be expressed in terms of Green's functions of degree less than or equal to $j+1$. 

It is quite easy to see that the values of the degree $1$ Green's functions appearing in the regulator are logarithms of algebraic numbers when evaluated at $CM$ points.  It turns out that the existence of indecomposable cycles immediately implies that the values of the degree $2$ Green's functions appearing in the regulator have values at $CM$ points which are logarithms of algebraic numbers. This is because at a $CM$ points $\tau$  the value can be realised as the regulator of an element of the group $H^5_{\M}(A_{\tau},\Q(3))$ -- which can be seen to be the logarithm of an algebraic number. In general, the existence of indecomposable cycles in $H^{2j+1}_{\M}(A^{j}_{\eta},\Q(j+1))$ similarly implies that values of the higher Green' s function of degree $\leq j+1$ are logarithms of algebraic numbers. 

This results in the following theorem (Theorem \ref{algebraicity}):
\begin{thm} Let $X$ be a modular curve and $G_2^X(z_1,z_2)$ be the higher Green's function of Gross and Zagier. Let $Z_{\Delta}^c=\sum (C_P,f_P)+(E_P,g_P)$ be the cycle constructed as above in the generic Kummer $K3$ surface of the modular curve  with boundary on points $\tau$ on $H_{\Delta}^c \cap X$. Let $y$ be a $CM$ point lying outside the support of $H_{\Delta}^c \cap X$. Then 
			
			$$\langle \reg(Z^c_{\Delta}),\eta_y  \rangle=\sum a_{\tau} G_2^X(\tau,y)=\log \left( \prod_{x \in C_P|_{y} \cap S_y} |f_P(x)|^{\ord_{C_P|_y \cap S_y}(x)} \prod_{x \in E_P|_{y} \cap S_y} |g_P(x)|^{\ord_{C_P|_y \cap S_y}(x)} \right)$$
			where $\sum a_{\tau} S_{\tau}$, where $S_{\tau}$ is the $CM$ cycle at $\tau$, is the boundary of $Z_{\Delta}^c$ up to decomposable elements. 
	
	In particular, since the points of $C_P \cap S_y$ and $E_P \cap S_y$  are algebraic and the functions $f_P$ and $g_P$  are defined over ${\bar \Q}$, this is the logarithm of an algebraic number.
\end{thm}

Gross, Kohnen and Zagier \cite{GKZ} conjectured that values at $CM$ points of certain higher Green's functions are logarithms of algebraic numbers. The Green's functions they considered were linear combinations of Green's functions of Heegner cycles obtained by relations among coefficients of modular forms of some weight $k$. 

They  also conjectured, and proved in the case of Heegner divisors on curves, that the Heeger cycles give coefficients of modular forms. Hence a `relation among coeficients of modular forms' suggests that there should be a `relation of rational equivalence' among Heegner cycles.

A relation among coefficients of modular forms of weight $k$ can be thought of as a weakly holomorphic modular form of weight $2-k$ and a `relation of rational equivalence among Heegner cycles of codimension $j$' is a motivic cycle in the group $H^{2j+1}_{\M}(A^j,\Q(j+1))$.

This suggests that there should be a link between weakly holomorphic modular forms and motivic cycles: A weakly holomorphic modular form should correspond to a special motivic cycle. This was speculated on in \cite{sree2000}. Borcherds \cite{borc} essentially does this in the case of Heegner divisors on modular and Shimura curves - he constructs functions with divisors on Heegner points from weakly holomorphic modular forms of weight $\frac{1}{2}$. 

In this paper we make a precise conjecture (Conjecture \ref{conj1}) using recent work of Brunier-Ehlen-Yang \cite{BEY} and earlier work of Zhang \cite{zhan} and Mellit \cite{mell}.
\begin{conj}

If $f$ is a weakly holomorphic modular form  for a group $\Gamma$ of weight $\frac{1}{2}-j$ there must exist a motivic cycle  $\Xi_f$ in $H^{2j+1}_{\M}(A^j,\Q(j+1))$, where $A$ is the universal family over $X(\Gamma)$ such that 
$$\langle \reg(\Xi_f),\w_{\tau}\rangle=\Phi_j(f,\tau)$$
where  $\Phi_j(f,\tau)$ is a Borcherds lift studied by Bruinier-Ehlen-Yang \cite{BEY} and $\w_{\tau}$ is a particular $(j,j)$ form in $A^j_{\tau}$.

\end{conj}
The case when $j=0$  is the theorem of Borcherds \cite{borc}. Our cycles correspond to the case $j=1$. The conjecture suggests that there should be special cycles in $H^3_{\M}(A,\Q(2))$ coming from weakly holomorphic modular forms of weight $-\frac{1}{2}$. We provide some evidence of this. 

Several cases of the Algebraicity Conjecture are known, sometimes conditionally, thanks to the works of Zhang \cite{zhan}, Mellit \cite{mell}, Viazovska \cite{viaz}, Zhou \cite{zhou} and finally Bruinier-Ehlen-Yang \cite{BEY}. The works of \cite{BEY} and Zemel \cite{zeme} imply certain relations which can be used to prove that Heegner cycles give coefficients of modular forms but they are  not quite relations of {\em rational equivalence}. An advantage of having a relation of rational equivalence is that we not only get the logarithm of the algebraic number but the number itself. Mellit showed this is some special cases as well. 

In Section \ref{computation}, we work out an example of the algebraicity result. We work out the intersection of the motivic cycle corresponding to $\Delta=5$ and a cycle in the fibre over the Humbert surface of invariant $4$. This gives a number which is algebraic over the function field of a component of $H_4$ and wll given an honest algebraic number when further restricted to a point lying on the interesection of a modular curve with the component of $H_4$. 

\subsubsection{The `modular' complex.}

Most constructions of exceptional cycles in the case of modular or Shimura varieties seem to come from `Shimura' sub-objects -- essentially cycles obtained from Shimura subvarieties,  their Hecke translates and the boundary components.  Ramakrishnan \cite{rama} formalised this in   the notion of the `modular complex' - a subcomplex of the Chow complex of a Shimura variety generated by these objects and the relations among them. While Ramakrishanan defined it only for Shimura varieties it can be extended to the universal families over them. Our cycles then lie in that complex. 

Our conjecture above can be formulated as a link between two short exact sequences - one coming from the modular complex and the other coming from the theory of mock modular forms or, equivalently, harmonic weak Maass forms. In our case the localization sequence gives a link between the motivic cohomology group of the generic fibre and different motivic cohomology groups of the special fibre and the total space. 
$$\cdots \longrightarrow H^{2j+1}_{\M}(A^j_{\eta},\Q(j+1)) \longrightarrow \bigoplus_{x \in X} H^{2j}_{\M}(A_x^j,\Q(j)) \longrightarrow H^{2j+2}_{\M} (W,\Q(j+1)) \rightarrow \cdots$$
Restricted to the modular subcomplex this can be seen to give a sequence involving the $CM$ cycles in the middle and indecomposable cycles as the first term. 

On the other hand, there is a well known sequence in the theory of weak Maass forms linking weakly holomorphic modular forms, harmonic weak Maass forms  and holomorphic modular forms,
$$ 0 \longrightarrow M^!_{\frac{1}{2}-j,\rho_L}(\Gamma') \longrightarrow H_{\frac{1}{2}-j,\rho_L}(\Gamma') \stackrel{\xi_{\frac{1}{2}-j}}{\longrightarrow} S_{\frac{3}{2}+j,\rho_L}(\Gamma') \longrightarrow 0.$$
Our conjecture asserts that the two sequences are related via the Borcherds lifts and regulator maps -- namely that the regulator of a motivic cycle evaluated on a $(j,j)$-form is the Borcherds lift of a weakly holomorphic modular form.  

One reason to formulate it this way is that here only the localization sequence involving $K_1$ and $K_0$ are involved - however there are sequence for higher $K$-groups as well. One might speculate there is some relation with modular forms there as well. 

\subsection{Other applications and remarks}
	
The construction of the motivic cycle   extends to the arithmetic case as well. If $A$ is an Abelian surface over a number field, we can construct cycles in $H^3_{\M}(A,\Q(2))$. These cycles can be used to prove that this group is not finitely generated. 

The Beilinson conjectures assert that for an Abelian suface over $\Q$ the rank of the integral motivic cohomology group  $H^3_{\M}(A,\Q(2))_{\ZZ}$ should be $3-b_1$ where $b_1$ is the rank of the group of codimension one cycles modulo homological equivalence. For instance in the case of the product of two non-isogenous elliptic curves over $\Q$, $b_1=2$ and he showed that there is $1$ motivic cycle using modularity of elliptic curves. The cycles we construct are {\em not} integral, but our construction suggests that there may be more than one cycle with the same boundary and perhaps considering their difference may give non-trivial integral cycles.

The actual construction is done on the Kummer $K3$ surface $\TK_A$ associated to the Abelian surface $A$. As mentioned above this can be realised as the double cover of $\CP^2$ ramified at six lines tangent to a conic.   One expects that much of the construction can be extended to the case of $K3$s which are double covers $\CP^2$ ramified at a sextic. 

\subsection{Outline of paper.} 

The outline of the paper is as follows. We first define the motivic cohomology groups that are of interest to us. Then we describe the construction of the cyles the case of products of elliptic curves followed by the general case. We then discuss the algebriacity results and the connection with the conjecture of Gross-Kohnen-Zagier \cite{GKZ} and the recent work of Bruinier-Ehlen-Yang \cite{BEY}. Finally we describe the modular complex and some other related questions. Along the way we work out an example of the construction of the motivic cycle and apply it to prove algebraicity in a special case.  

\vspace{\baselineskip}

\section *{Acknowledgements.} I would like to thank Srinath Baba, Souvik Bera, Jishnu Biswas, Spencer Bloch, Patrick Brosnan, Jan Bruinier, Sudeepan Datta, Najmuddin Fakhruddin, Volker Genz, Bernhard Heim, Satoshi Kondo, Anton Mellit, Arvind Nair, Kapil Paranjape, Tanay Phatak, G.V. Ravindra,  Subham Sarkar, Chad Schoen, Ronnie Sebastian,  B Sury and  Maryna Viazovska, among many others, who have listened to me talk about this question for years and have provided useful suggestions and comments.

\section{Elements of the Motivic Cohomology.}

\subsection{The group $H^3_{\M}(A,\Q(2))$.}

Let $X$ be a surface over a field $K$. An element of the group $H^3_{\M}(X,\Q(2))$ has the following presentation 
$$\sum (C_i,f_i)$$
where $C_i$ are curves on $X$ and $f_i$ functions on $C_i$ satisfying the cocycle condition 
$$\sum_i \div(f_i)=0$$ 

One way of constructing such elements is by considering any curve $C$ on $X$ with a constant function $a \in K^*$. Then $\div(a)=0$ and hence $(C,a)$ is such an element. There is a product structure on motivic cohomology. If $L/K$ is a finite extension,  one has a map 
$$ \bigoplus_{L/K} H^2_{\M}(X_L,\Q(1)) \otimes H^1_{\M}(X_L,\Q(1)) \longrightarrow H^3_{\M}(X_L,\Q(2)) \stackrel{Nm^L_K}{\longrightarrow} H^3_{\M}(X,\Q(2))$$
and the elements $(C,a)$ lie in the image of this map. The image of this map is called the group of  {\em decomposable} cycles and will be denoted by $H^3_{\M}(X,\Q(2))_{dec}$. The quotient group 
$$H^3_{\M}(X,\Q(2))_{ind}=H^3_{M}(S,\Q(2))/H^3_{\M}(S,\Q(2))_{dec}$$
is called the group of {\em indecomposable} cycles. In general it is not clear how to construct non-trivial elements of this group. 

There are a few constructions of indecomposable elements.  Bloch \cite{bloc} constructed them  on $X_0(37) \times X_0(37)$ and  Beilinson \cite{beil} generalised this to  products of modular curves. Collino constructed cycles in the Jacobian of a genus $2$ curve. There is a general construction inspired by Bloch's in the case when $X$ is a product of curves which have a set of points such that any divisor of degree $0$ supported on them is torsion in the Jacobian and all the constructions above are special cases of this. 

Another way of constructing cycles is the following - which is the method we will use later so we label it a proposition. 
\begin{prop}  Let $Q$ be a {\em nodal rational curve} on $X$ with node $P$. Let $\nu:\tilde{Q}\rightarrow Q$ be its  normalization in the blow up $\tilde{X}$ of $X$ at $P$  The strict transform $\tilde{Q}$ meets the exceptional fibre $E_P$ at two points $P_1$ and $P_2$. Both $\tilde{Q}$ and $E_P$ are rational curves. Let $f_P$ be the function with $\div(F_P)=P_1-P_2$ on $\tilde{Q}$ and similarly let $g_P$ be the function with $\div(g_P)=P_2-P_1$ on $E_P$. Then 
	$$(\tilde{Q},f_P)+(E_P,g_P)$$
	is an element of $H^3_{\M}(\tilde{X},\Q(2))$. 
	\label{construction}
	\end{prop}
This was used by Lewis and Chen \cite{lech} to prove the Hodge ${\mathcal D}$-conjecture for certain $K3$ surfaces  and by me \cite{sree2014} to prove an analogue of that in the non-Archimedean case for Abelian surfaces.

\subsection{The long exact localization sequence.} 

If $\XX$ is a surface over a base $S$ and $X_{\eta}$ is the generic fibre  then one has a {\em long exact localization sequence} linking the motivic cohomologies of $\XX$, $X_{\eta}$ and the special fibres $X_x$,
$$\cdots \longrightarrow H^3_{\M}({\XX},\Q(2)) \longrightarrow H^3_{\M}(X_{\eta},\Q(2)) \stackrel{\partial}{\longrightarrow} \bigoplus_{x \in S^1} H^2_{\M}(X_x,\Q(1)) \longrightarrow H^4_{\M}(\XX,\Q(2)) \longrightarrow \cdots $$
where $S^1$ are the cycles of codimension $1$ in $S$. 
The boundary map is the following. If $\sum (C_i,f_i)$ is a cycle in $H^3_{\M}(X_{\eta},\Q(2))$, let $\CC_i$ be the closure of $C_i$ in $\XX$. Then 
$$\partial (\sum_i (C_i,f_i))=\sum_i \div_{{\mathcal C}_i}(f_i).$$
The cocycle condition implies that $\overline {\div_{C_i}(f_i)}=0$ so the boundary is supported on `vertical' cycles in the fibres $X_x$. 

In the case of a decomposable element this boundary is easy to compute. If $(C,a)$ is a decomposable cycle with $a \in K(S)^*$ and $C$ a curve in the generic fibre  then 
$$\partial ((C,a))=\sum_x \ord_x(a) C_x$$
where $C_x$ is the restriction of  the closure of $C$ to $X_x$. This shows that the only cycles appearing in the boundary of of decomposable cycles are the  restrictions of the closures of codimension one cycles on $X_{\eta}$ to the fibre $X_x$. There are instances when the group $CH^1(X_x)=H^2_{\M}(X_x,\Q(1))$ is larger than $CH^1(X_{\eta})=H^2_{\M}(X_{\eta},\Q(1))$ and if one obtains these `extra' cycles in the boundary of a motivic cycle then that cycle is necessarily indecomposable. 

The group $H^4_{\M}(\XX,\Q(2))$ is $CH^2(\XX)$, hence the motivic cohomlogy group $H^3_{\M}(X_{\eta},\Q(2))$ is the `space of relations of rational equivalence between codimension $2$ cycles supported in the fibres'. More generally, the group $H^{2p-1}_{\M}(X_{\eta},\Q(p))$, where $X_{\eta}$ is of dimension $2p-2$,  is the space of relations of rational equivalence  between codimension $p$ cycles supported in fibres. For instance, when $p=1$ this is just the space of relations between points -- namely functions on $X$. 

\section{The Siegel Modular threefold and Humbert surfaces.}

We are interested in the case when the base variety is the Siegel modular threefold $S_2(2)$ parameterizing Abelian surfaces with full level $2$ structure. In this case for a general Abelian surface $A$, the rank $\rho$ of the N\'{e}ron-Severi group $NS(A)$  is 1. There are special subvarieties of $S_2(2)$ corresponding to points where  $\rho \geq 2$ called {\em Humbert surfaces}. Equivalently, these are  the images of Hilbert modular surfaces in the Siegel modular threefold. We will construct cycles in the generic fibre of the universal family  which have boundaries supported on components of  these subvarieties and, further, the boundary consist of the extra cycles determined by the subvarieties.  From the remarks above these cycles are necessarily indecomposable.

\subsection{Humbert surfaces.} Let $(A,\Theta)$ be a principally polarised Abelian surface. If $D \in NS(A)$ following \cite{kani} define the Humbert norm to be
$$H(D)=(D,D)_H=(D.\Theta)^2-2D^2.$$
From the Hodge Index Theorem one knows that the intersection pairing is negative definite on the orthogonal complement of the class of the polarization. The Humbert pairing is simply the negative of the intersection pairing on that and is therefore positive definite. The Humbert surface $H_{\Delta}$ is the set
$$H_{\Delta}=\{z \in S_2(2)| \exists D \in NS(A_z) \text{ such that } H(D)=\Delta\}$$
 If $A_z$ is an Abelian surface then \cite{bess},
 $$A_z \text{ has real multiplication by } \ZZ[\frac{\Delta+\sqrt{\Delta}}{2}] \Leftrightarrow z \text{ lies on } H_{\Delta}$$

 The surfaces $H_{\Delta}$ are the images of Hilbert modular surfaces in $S_2(2)$. 
 For instance, an Abelian surface is a product of elliptic curves if its moduli point lies on $H_1$. Here $\Delta=1$ and the endomorphism ring is the degenerate case of $\ZZ \oplus \ZZ$. More generally $H_{n^2}$ corresponds to the moduli of Abelian surfaces $A=J(C)$ where there is map $p:C \longrightarrow E$ of degree $n$ from the polarisation $C$, which is a genus $2$ curve, to an elliptic curve $E$ \cite{kani}.

 In \cite{coll} Collino constructed an indecomposable cycle in the motivic cohomology of the generic Abelian surface which has a boundary on a component $H_1$. In general, on a moduli of Abelian surfaces one expects there to be motivic cycles whose boundary is supported on special submoduli, where the Picard rank increases.  In this paper we construct examples of such cycles in the cases of some families of $\Delta$. 
 
 \section{Motivic cycles on  Abelian surfaces.}

\subsection{The case of products of elliptic curves.}

As a first case, we consider products of elliptic curves. Let $X$ be the moduli of products of elliptic curves with full level $2$ structure as well as suficient level structure to ensure there is a universal family.  Let $S$ be $X \times X$. Let $A_{\eta}$ be the generic fibre of the universal family over $S$. Over a point $(\tau_1,\tau_2)$ the fibre is $A_{\tau_1,\tau_2}=E_{\tau_1} \times E_{\tau_2}$. We construct cycles in $H^3_{\M}(A_{\eta},\Q(2))$ such that their boundary is supported on those points $(\tau_1,\tau_2)$ where  $E_{\tau_1}$ and $E_{\tau_2}$ are isogenous. This happens on  modular curves lying on $X \times X$. Further, the boundary consists of the graph of the isogeny. This construction is due to Spiess \cite{spie}  in the local case so we will largely follow him.

We use the following theorem of Frey and Kani \cite{frka}. This theorem asserts that under some conditions a product of elliptic curves is isogenous to the Jacobian of a stable genus $2$ curve.

\begin{thm} Let $E_1$ and $E_2$ be two elliptic curves over a noetherian scheme $S$. Let $n$ be an odd integer which is invertible on $S$ and $\phi:E_1[n] \rightarrow E_2[n]$ an anti-isometry with respect to the Weil pairing $e_n$ - i.e $e_n(\phi(x),\phi(y))=e_n(x,y)^{-1}$. Let $H=Graph(\phi)$ and $J=(E_1 \times E_2)/H$ with $p:A \rightarrow J$ the projection. Then there exists a proper flat morphism $q:C \rightarrow S$ and a closed immersion $j:C \rightarrow J$ such that 
	
	\begin{itemize}
		\item $C$ is a stable genus $2$ curve in the sense of Deligne-Mumford. 
		\item $j(C)$ is fixed under the involution $-1:J \rightarrow J$: $-j(C)=j(C)$.
		\item The line bundle associated to $j(C)$ is a principal polarization and the composite map 
		$$A \stackrel{p}{\longrightarrow} J \stackrel{\lambda_{j(C)}}\longrightarrow \hat{J} \stackrel{\hat{p}}{\longrightarrow} A$$
		is given by multiplication by $n$ where $\lambda_{j(C)}$ is the isomorphism $J\longrightarrow \hat{J}$ induced by the principal polarization. 
		\item The composite maps
		$$\pi_i:C \stackrel{j}{\longrightarrow} J \longrightarrow \hat{J} \stackrel{\hat{p}}{\longrightarrow} E_1 \times E_2 \stackrel{p_i}{\longrightarrow} E_i$$
		are finte maps of degreen $n$ for $i \in \{1,2\}$.   
	\end{itemize} 
	
	\label{freykani}
	
\end{thm}

\begin{proof} \cite{frka}, \cite{spie} \end{proof}

To find such an anti-isometry we do the following. Let $E_x$ and $E_y$ be two elliptic curves such that there is an isogeny $h:E_x \rightarrow E_y$ of degree $(n-1)$. Let $\phi_h$ denote the induced map on $n$ torsion. Then $\phi_h$ induces an anti-isometry $E_x[n] \rightarrow E_y[n]$ as
$$e_n(\phi_h(a),\phi_h(b))=e_n(a,\phi_{\hat{h} \circ h}(b))=e_n(a,(n-1)b)=e_n(a,-b)=e_n(a,b)^{-1}.$$
for all $a,b \in E_x[n]$. While $h$ need not extend to all $E_x$ and $E_y$ the anti-isometry $\phi_h$ can be lifted to an isometry $\tilde{\phi}_h:E_{\eta_1}[n] \longrightarrow E_{\eta_2}[n]$. 

From  Theorem \ref{freykani}, if $H$ is the graph of $\phi_h$ then $A_{\eta}/H$ is the Jacobian of a stable genus $2$ curve $C_{\eta}$.  Spiess \cite{spie}, Lemma 3.3, shows  this curve is generically irreducible. 

However, at those points $s=(x,y)$ in whose fibres where there is an isogeny of $h$ of degree $(n-1)$ and  $H$ is given by the graph of the isogeny, the genus $2$ curve splits in to the sum of two elliptic curves meeting at a point:
$$C_s=\Gamma^t_{-\hat{h}} \sqcup_0 \Gamma_{h} \simeq E_x \sqcup_0 E_y$$
and the maps $\pi_x:C_s \longrightarrow E_x$ is given by $\pi_x=id \sqcup_0 -\hat{h}$ and $\pi_y:C_s \longrightarrow E_y$ is given by $\pi_y=h \sqcup_0 id$.

The set of such points $s$ is given by  a component of a modular curve on $X \times X$ and for every such isogeny we can get a corresponding genus smooth irreducible genus $2$ curve in the generic fibre. Let $X(h)$ denote that component. In the fibre over $X(h)$ one has the cycles $\Gamma_h$ and $\Gamma_{-hat{h}}$. 

Following Spiess, we use this to construct a motivic cycle. We first remark since we are assuming full level two structure  the Wieirstrass points are defined over the function field. There are $6$ points on $C$ and in the special fibre, three lie on one component and three lie on the other. The fourth Wieirstrass point on the each elliptic curve is given by the point of intersection. Let $P$ be a point restricting to one component, say $E_x$ and $Q$ a point restricting to $E_y$. There exists a function $f$ on $C$ with 
$$\div(f)=2P-2Q.$$ 
Let $P_i$ and $Q_i$ denote the images of $P$ and $Q$ under $\pi_i$. There exists a function $f_i$ on $E_i$ with 
$$\div(f_i)=2P_i-2Q_i$$
The element 
$$Z_h=(C,f)+(E_1 \times P_2,f_1^{-1})+(Q_1 \times E_2,f_2^{-1})$$ 
is an element of the motivic cohomology group $H^3_{\M}(A_{\eta},\Q(2))$. We next compute its boundary.

\begin{thm}[\cite{spie}] The boundary of the element $Z_h$ is, up to a decomposable element $m \Gamma_h$ in the fibre over $X(h)$ for some $m \neq 0$.  
\end{thm}

To determine the boundary more precisely, if the motivic cycle has a boundary, it implies that there is an isogeny $h$  of degree $(n-1)$ from $E_1$ to $E_2$. In this case $\Theta=E_1 \times 0  + 0 \times E_2$. If $\Gamma_h$ denotes the graph of the isogeny then
$$H(\Gamma_h)=(\Gamma_h.\Theta)^2 -2(\Gamma_h,\Gamma_h)$$
$(\Gamma_h,\Theta)=(\Gamma_h,E_1 \times 0)+(\Gamma_h, 0 \times E_2)=(n-1)+1=n$. $(\Gamma_h,\Gamma_h)=0$ from the genus formula as $\Gamma_h$ is an elliptic curve. Hence 
$$H(\Gamma_h)=n^2.$$

The moduli of products of elliptic curves is a double cover of  the Humbert surface $H_1$ under the involution given by $(x,y) \rightarrow (y,x)$. Hence the cycle we have degerates on a curve $C$ on $S$ which maps to a component of $H_1 \cap H_{n^2}$. This intersection is a union of graphs of Hecke corresopondences of the form $T_{\frac{n^2-s^2}{4}}$.

\subsection{The general case.}

On the Humbert surface $H_1$, it is clear what the new cycles in the fibres are -- if $A=E_1 \times E_2$ then the cycles $E_1 \times 0$ and $0 \times E_2$, both of which have Humbert norm $1$. As shown above when they are isogenous the graph of the isogeny gives a cycle of Humbert norm $n^2$ where $n$ is the degree of the isogeny. In general, on $H_{\Delta}$, it is not clear how to describe the new cycles in the fibres.  

Humbert, and more recently Birkenhake-Wilhelm \cite{biwi}, gave a description of these cycles for several classes of $\Delta$. In order to understand their description, we need the following.

\subsubsection{The Kummer construction.}

Let $(A,\LL)$ be a principally polarised  Abelian surface, where $\LL=\LL_{\Theta}$ is the line bundle corresponding to an irreducible genus $2$ curve $\Theta$ on $A$. The {\em Kummer surface} $K_A$ is the image of $A$ in $\CP^3$ under the map 
$$\phi=\phi_{\LL^2}: A \longrightarrow \CP^3=\CP(H^0(A,\LL^2))$$
It is well known that $K_A$ isomorphic to $A/\pm 1$ and has $16$ nodes corresponding to the $16$ 2-torsion points on $A$. The blow up of $K_A$ at those $16$ points is a $K3$ surface $\tilde{K}_A$ which we will call the {\em Kummer K3 surface}. Let $\nu: \TK_A \longrightarrow K_A$ be the blow up.  
 
One has another map 
$$\pi:K_A \longrightarrow \CP^2$$
obtained by projection from the image of $0$. This is a double cover ramified at $6$ lines tangent to a conic. We call the configuration of $\CP^2$ with the six lines $l^i$  the {\em Kummer plane}, $\CP^2_A=(\CP^2,l^1,\dots,l^6)$. These lines meet at $15$ points $q^{ij}=l^i \cap l^j$ corresponding to the images of the $15$ non-zero $2$-torsion points. By abuse of language we will refer to these points as "two-torsion" points as well. 

One has a diagram 

$$\begin{CD} 
 	@. \TK_A   \\
@. @VV\nu V \\
A @>\phi>> K_A @>\pi>> \CP^2_A
\end{CD}$$

Conversely, given a $\CP^2$ with $6$ lines tangent  to a conic one can recover the Abelian surface. Hence the moduli of principally polarised Abelian surfaces is the same as the moduli of six lines in $\CP^2$ tangent to a conic. The idea of Humbert and Birkenhake-Wilhelm is to describe the extra cycle in terms of this data. 

\subsubsection{The Theorem of Birkenhake-Wilhelm.}

We first recall the theorem of Birkenhake-Wilhelm \cite{biwi}:

\begin{thm}[Birkenhake-Wilhelm] Let $(A,\LL_{\Theta})$ be a principally polarized Abelian surface with a line bundle $\LL_{\Delta}$ of invariant $\Delta$. Let $S=\prod_{i=1}^6 l^i$ be the (degenerate) sextic given by the product of ramified lines in the Kummer plane $\CP^2_A$. Let $m$  be a natural number and $k \in \{4,6,8,10,12\}$. Then there exists a rational curve $Q_{\Delta}$ on $\CP^2_A$ which passes through some of the points $q^{ij}$ with no singularities at those points and meets $S$ at the remaining points with even multiplicity. One has the following cases:
	\vspace{\baselineskip}
	\begin{center}
		\begin{tabular}{||l|c|c|c||}
			\hline
			Case  & $\Delta$ & $ d=deg(Q_{\Delta})$ & No. of points $q^{ij}$ \\
			\hline
			$I.$ & $8m^2 + 9 - 2k$ & $2m$ & $k-1$\\
			$II.$ & $8m(m+1)+9-2k$ & $2m+1$ & $k$\\ 
			$III.$ & $8m^2+8-2k$ & $2m$ & $k$ \\
			$IV.$ &$8m(m+1)+12-2k$ & $2m+1$ & $k-1$\\
			$V.$ & $m^2$ &$m-1$ & $3$\\
			\hline
		\end{tabular}
	\end{center}
	\vspace{\baselineskip}
	Further, $Q_{\Delta} = \pi \circ  \phi(D_{\Delta} )$ where $D_{\Delta}$ is a curve on $A$ which lies in the linear system of divisors of a line bundle $\LL$ of the form $\LL_{\Theta}^a \otimes  \LL_{\Delta}^b$  with $b \neq 0$. In particular, the class of $D_{\Delta}$ is not a multiple of the class of the principal polarization.

\end{thm}

Note that the theorem says that for an Abelian surface $A$ corresponding to a  point $z_0$ on $H_{\Delta}$ there is a $Q_{\Delta}$ in the corresponding $\CP^2_A$. The corresponding curve on the Abelian surface  $D_{\Delta}$ will be defined only in the universal family over the component on which $z_0$ lies. There will be a different curve defined over the other components. A different choice of points $q^{ij}$ and points of even multiplicity will correspond to a curve defined over a different component of $H_{\Delta}$.

To obtain the curve $D_{\Delta}$ we make the following observation. 

\begin{lem} If $\pi:C \longrightarrow Q$ is a double cover of a rational curve $Q$ for which all ramification points are singular. Then the normalization $\tilde{C}$ is the union of two rational curves meeting at a point. 
	
\end{lem}
\begin{proof} The map $\pi$ induces a map $\tilde{\pi}:\tilde{C} \longrightarrow \CP^1$. Since all the ramification points of $\pi$ are singular, the map $\tilde{\pi}$ is an {\em unramified} double cover of $\CP^1$. There are no irreducible unramified double covers of $\CP^1$ hence this is a union of two curves $\TC_1$ and $\TC_2$ meeting at some points. We also have 
	$$0=g(\TC)=g(\TC_1)+g(\TC_2)+(\TC_1,\TC_2)-1$$
	therefore $(\TC_1,\TC_2)=1$ and the two components meet at a point. 
\end{proof}

We apply this to the rational curve $Q_{\Delta}$. The map $\pi:K_A \rightarrow \CP^2_A$ induces a double cover from the preimage $C$ of $Q_{\Delta}$ to $Q_{\Delta}$. From the lemma one can see that the curve $C$ will have two components $C_1$ and $C_2$. Either one of them will pull back to $A$ to give the extra cycle $D_{\Delta}$. The two curves $C_1$ and $C_2$ are conjugate in the sense that the involution induced by $\pi$ corresponds to Galois conjugation in $Q(\sqrt{\Delta})$.

\subsubsection{The construction of cycles in the group $H^3_{\M}(A,\Q(2))$.}

The Theorem of Birkenhake-Wilhelm above says that if $\Delta$ is as in the theorem and  $A$ has multiplication by $\Q(\sqrt{\Delta})$ there is an exceptional rational curve $Q_{\Delta}$ of  degree $d$  meeting the sextic  $S$  only at points of even multiplicity or nodes of the sextic.

The idea here to show there exists a rational curve satisfying slighly weaker conditions for {\em all} Abelian varieties which restricts to $Q_{\Delta}$ when the moduli point of $A$ lies on the component of $H_{\Delta}$ corresponding to $Q_{\Delta}$. 

There is a well known theorem in enumerative geometry which states the following \cite{zing}:

\begin{thm} 
	\label{enumgeom}
	Let $d>0$ be a natural number. Let $n_d$ be the number of rational curves of degree $d$ passing though $3d-1$ points in general position. Then $n_d$ is a finite, non-zero number. 
\end{thm}

The precise number $n_d$ is the celebrated formula  of Kontsevich-Manin \cite{koma} and Ruan-Tian \cite{ruti}. For small $d$ it is classical: $n_1=1,n_2=1,n_3=12$.

More generally, let $k, m_1,\dots m_r$ be non-negative integers such that $k+\sum m_i=3d-1$. Then there are finitely many rational curves of degree $d$ passing through $k$ points and meeting $3d-1-k$ lines with  multiplicity $2m_i$, since that is a special case of the above theorem.  Let $n_{d; k, m_1, \dots, m_r}$ denote this number.  For instance there are two conics passing through $4$ points and tangent to a given line, so $n_{2; 4, 1}=2$ and one can see $n_{2; 3, 1, 1}=4$.

Armed with this we have the following theorem.

\begin{thm} 
	\label{motiviccycle} Let  $\Delta$ be as in the theorem of Birkenhake-Wilhelm. There exists a motivic cycle $\Xi^c_{\Delta}$ in $H^3_{\M}(A_{\eta},\Q(2))$ where $A_{\eta}$ is the generic abelian surface over the Siegel modular threefold. This cycle is defined in the fibres outside the component $H^c_{\Delta}$ of the Humbert surface $H_{\Delta}$ corresponding to $Q_{\Delta}$.
	
\end{thm}

\begin{proof}
	
Let $A_z$ be an Abelian surface depending on a parameter $z$ in the Siegel modular threefold. Corresponding to $A_z$ there are the lines $l^i_z$ and the points $q^{ij}_z$ in $\CP^2_{A_z}$. These vary smoothly with $z$. Let 
$$S_z=\prod_{i=1}^{6} l^i_z $$
be the sextic given by the product of the six lines. 

If $z_0$ lies on $H_{\Delta}^c$ the theorem of Birkenhake-Wilhelm says that there is a rational curve $Q_{\Delta}$ of degree $d$ meeting $S_{z_0}$ at some of the points $q^{ij}_{z_0}$ as well as some other points $t^k_{z_0}$ of even multiplicity $2m_k$. In general $(Q_{\Delta}, S_{z_0})=6d$ but since every point is of multiplicity at least $2$ there are at most $3d$ distinct points. 

Let $q^{i_0j_0}_{z_0}$ be one of the two torsion points lying on $Q_{\Delta}$. The completement of this in $Q_{\Delta} \cap S_{z_0}$ detemines $3d-1$ points. From Theorem \ref{enumgeom}, for any $z$ there exists a rational curve $Q_z$ of degree $d$ determined by the condition that it passes through the points  $q^{ij}_z$ for the $(ij) \neq (i_0j_0)$ and meets the lines $l^k$ at  points of $t^k_z$ with multiplicity $2m_k$. Further, this family varies smoothly in $z$  one has family of curves $Q_z$  which restricts to $Q_{\Delta}$ when $z=z_0$. 

Apart from the points common with $Q_{\Delta}$, $Q_z$ will meet $l^{i_0}_z$ at a point $s^{i_0}_z$ and $l^{j_0}_z$ at a point $s^{j_0}_z$, both with multiplicity $1$. Hence the double cover $\pi:C_z \longrightarrow Q_z$ induced from the map $\pi:K_{A_z} \longrightarrow \CP^2_{A_z}$ is ramified at the points $q_z^{ij}$ lying on $Q_z$, the points $t^k_z$, all of which are singular and finally the two points $s^{i_0}_z$ and $s^{j_0}_z$. The normalization, $\TC_z$ is therefore a double cover of $\CP^1$ ramified at $2$ points -- hence it is  an {\em irreducible} rational curve. 

A motivic cycle in $H^3_{\M}(\TK_{A_z},\Q(2))$ is given by a sum 
$$\sum (C_i,f_i)$$
such that $\sum \div(f_i)=0$. We construct such as cycle as follows. Let $\BC_z$ be the strict transform of $C_z$ in $\TK_{A_z}$. Let $q^{ij}_z$ be a point lying on $Q_z$. Such a point exists as from the theorem of Birkenhake-Wilhelm one can see that  $Q_{\Delta} \cap S_{z_0}$ has at least $3$ points of the form $q^{ij}_{z_0}$ and we have discarded only one to get the rational curve. Let $p_z$ be the point on $C_z$ lying over it. On $\TK_{A_z}$, which is obtained by blowing up all the points $q^{ij}_z$, there are two points $p_{1z}$ and $p_{2z}$ on  $\BC_z$ lying over $p_z$. The points lie on the interesection of the exceptional fibre $E_{p_z}$ with $\BC_z$. 

Let $f_{p_z}$ be the function on the rational curve $\BC_z$ with divisor  $\div(f_{p_z})=p_{1z}-p_{2z}$. To determine the function precisely we require  $f_{p_z}(s^{i_0}_z)=1$.  Similarly, let $g_{p_z}$ be a function on $E_{p_z}$ with $\div(g_{p_z})=p_{2z}-p_{1z}$. Then 
$$Z_{\Delta,z}^c=(\BC_z,f_{p_z})+(E_{p_z},g_{p_z})$$
is an element of $H^3_{\M}(\TK_{A_z},\Q(2))$. 

This construction works as long as $C_z$ is irreducible. This is the case as long as $s^{i_0}_z$ and $s^{j_0}_z$ do not coincide. When they coincide they agree with $q_{z}^{i_0j_0}$ and $Q_z=Q_{\Delta}$. This implies that $z$ lies on a component of $H_{\Delta}$.

Since $Z_{\Delta_z}^c$  is defined outside the complement of a codimensional  $1$ subvariety, it can therefore be thought of as an element of $H^3_{\M}(\TK_{A_{\eta}},\Q(2))$. Using the maps $\phi$ and $\nu$ let 
$$\Xi_{\Delta}^c=\phi^*(\nu_*(Z_{\Delta}^c))$$ 
This is a cycle in $H^3_{\M}(A_{\eta},\Q(2))$. 

\end{proof}

\begin{rem} We can also use a point of even multiplicity to `deform' the rational curve $Q_{\Delta}$. This gives cycles which are defined outside $H_{\Delta}^c$ as well as possibly other components as the $C_z$ can become reducible when $z$ lies on $H_{\Delta'}^c$ for some $\Delta' \leq \Delta$. For instance, if $Q_{\Delta}$ meets a line at a point of multiplicity $4$ then the deformed curve will meet at a point of multiplicity $2$ and two smooth points. There may be a locus where the two points meet but they do not coincide with the point of multiplcity $2$ and that could correspond to a different $\Delta$. As we show below the boundary of this cycle is supported on the locus there $C_z$ becomes reducible - so such cycles will have boundary on different Humbert surfaces. 	
	\end{rem}

\begin{rem} It is not clear what role is played by the the finite number $n_d$ or its generalization $n_{d;k,m_1,\dots,m_r}$. In the case of an Abelian surface over a number field it is  plausible that different deformations of the curve $Q_{\Delta}$ could be used to construct two elements with the same boundary - hence their difference is in the `integral' motivic cohomology. In that case, however, it is not clear if they are non-trivial. 
	
	\end{rem}

The next theorem shows that the boundary of the motivic cycle can be expressed in terms of the extra components determined by $Q_{\Delta}$.

\begin{thm} Let $Z^c_{\Delta}$ be the cycle constructed above and let 
	$$\bar{C}_z=\bar{C}_z^1 \cup \bar{C}_z^2$$
	when $z$ lies in the component $H_{\Delta}^c$ determined by the condition that $s^{i_0}_z=s^{j_0}_z=q^{i_0j_0}_{z}=q$.  
	
	Then the boundary  
	$$\partial(Z^c_{\Delta})=a(\bar{C}_z^1-\bar{C}_z^2)$$
	for some $a \neq 0$. In particular, $Z^c_{\Delta}$ is {\em indecomposable}.
	
	\label{boundary}
	
\end{thm}

\begin{proof}

	Let $H=H_{\Delta}^c$. To compute the boundary, we  compute the divisor of $f_{p_{\eta}}$ on the closure of  $\BC_{\eta}$.  In the fibre over  $H$, the strict transform $\BC_H=\BC_{\eta}|_H$ splits into components $\BC_H^1 \cup \BC^2_H$. 
	
	Let $\iota$ denote the involution determined by the double cover. The points $p_{1,H}$ and $p_{2,H}$ lie on different components of $\BC_H$ as $\iota(p_{1,H})=p_{2,H}$ and $\iota(\BC_H^1)=\BC_H^2$.  We can assume without loss of generality that $p_{1,H}$ lies on $\BC^1_H$. If $E_{p_H}$ is the exceptional fibre then $\iota(E_{p_H})=E_{p_H}$. 
	
	We have  
	$$\div_{\overline{\BC}}(f_{p_{\eta}})=\mathcal{H}+a \tilde{C}_H^1+b \tilde{C}_H^2$$
	where $\mathcal{H}$ is $\overline{\div_{\BC}(f_{p_{\eta}}})$ is the closure of the horizontal divisor $p_{1,H}-p_{2,H}$. 
	
	We first claim $a=-b$. To see this observe that the function $f^{\iota}_{p_{\eta}} = f_{p_{\eta}} \circ \iota$ satisfies 
	$$\div_{\BC}(f^{\iota}_{p_{\eta}})=-\div_{\BC}(f_{p_{\eta}})$$
	Therefore $$f^{\iota}_{p_{\eta}}=\frac{k_{\eta}} {f_{p_{\eta}}}$$
	for some constant $k_{\eta}$. Since, by choice, $f_{p_{\eta}}(s^{i_0}_{\eta})=1$ and $\iota(s^{i_0}_{\eta})=s^{i_0}_{\eta}$ we have $f^{\iota}_{p_{\eta}}(s^{i_0}_{\eta})=f(s^{i_0}_{\eta})=1$. Hence $k_{\eta} \equiv 1$ and we have
	$$\div_{\overline{\BC}}(f^{\iota}_{p_{\eta}})=-\div_{\overline{\BC}}(f_{p_{\eta}})=-\mathcal{H}-a \BC_H^1-b \BC_H^2$$
	On the other hand  $\div \circ \, \iota = \iota \circ \div$ and  $\iota(\BC_H^1)=\BC_H^2$  so one has 
	$$\div_{\overline{\BC}}(f^{\iota}_{p_{\eta}})=-{\mathcal H} + b \BC_H^1 + a\BC_H^2$$
	Comparing coefficients shows that $a=-b$. Hence 
	$$\div_{\overline{\BC}}(f_{p_{\eta}})=\mathcal{H}+a ( \BC_H^1 - \BC_H^2).$$
	We now claim $a \neq 0$. Suppose $a=0$. Then $f_{p_{\eta}}|_{\BC_{H}^1}$ is a function on $\BC_{H}^1$ hence its divisor has degree $0$. Therefore
	$$0=\deg(\div_{\overline{\BC}}(f_{p_{\eta}}|_{\BC_H^1}))=(\div_{\overline{\BC}}(f_{p_{\eta}}),\BC_H^1)=( {\mathcal H},\BC_H^1)=1$$
	as ${\mathcal H} \cap \BC_H^1=\{p_{1,H}\}$ so $({\mathcal H},\BC_H^1)=1$. So we have a contradiction and hence $a \neq 0$. 
	
	The other component of the motivic cycle is  $(E_{p_{\eta}},g_{p_{\eta}})$. The closure of the curve $E_{p_{\eta}}$ remains irreducible over $H$. Hence 
	$$\div_{\overline{E_{p_{\eta}}}} (g_{p_{\eta}})=-\mathcal{H}$$
	
	Adding the two gives 
	$$\partial(Z^c_{\Delta})=a(\BC_H^1-\BC_H^2)$$ 
	for some $a\neq 0$. 
	
	In the case when there is a decomposable element $Z'=(\BC,g)$ with boundary $b(\BC_H^1+\BC_H^2)$ for some $b \neq 0$  we have $a=\frac{1}{2}$. This is because we can consider $Z''=bZ^c_{\Delta}-aZ'$. Computing the boundary as above we have 
	$$\div_{\overline{\BC}}(f_{p_{\eta}}^{b}g^{-a})=b{\mathcal H}+2ab \BC_H^1$$
	and 
	$$0=\deg(\div (f^b_{p_{\eta}}g^{-a}|_{\BC_H^2}))=({\mathcal H},\BC^2_{H})+2a(\BC_H^1,\TC_H^2)=-b+2ab$$
	so since $b\neq 0$,  $a=\frac{1}{2}$.
	
\end{proof}

Since the boundary consists of cycles that are not defined in the generic fibre it is indecomposable. The cycle $\Xi_{\Delta}^c$ has boundary $a \left(D^1_H-D^2_H\right)$. 

We now show that one can assume that the boundary only has `exceptional cycles'.

\begin{lem} Let $X$ be a subvariety of the Siegel modular threefold. Let $W \rightarrow  X$ be the universal family which we assume exists.  Let $\Xi$ be a motivic cycle in $H^3_{\M}(A_{\eta},\Q(2))$, where $A_{\eta}$ is the generic fibre of the universal family. Assume the Picard number of $A_{\eta}$ is $r$ generated by $D^i_{\eta}, 1 \leq i \leq r$. Note that $r$ can be at most $3$ and the maximal Picard number of an Abelian surface in characteristic $0$ is $4$.  Suppose the boundary is 
	$$	\partial(\Xi)=\sum_x b_x S_x+  a^1_{x}D^1_x + \dots + a^r_{x} D^{r}_x $$
	where  $D^i_x=D^i_{\eta}|_x$ and $S_x$ is a cycle which exists only on the special fibre $A_x$ and is such that $(S_x.D^i_x)=0$  for all $i$.
	
	Then there exist decomposable elements of the form $(D_{\eta}^i,g^i)$ such that if $\tilde{\Xi}=\Xi - \sum (D_{\eta}^i,g^i)$ then 
	$$\partial(\tilde{\Xi})=\sum b_x S_x$$
	
	\label{decomp}
	
\end{lem}

\begin{proof} Let $M_i$ be cycles in $CH^1(W) \otimes \Q=H^2_{\M}(W,\Q(1))$ such that $(M_i,D^j_{\eta})=\delta_{ij}$ and $(M_i|_{A_x}.S_x)=0$.  $M_i|_{A_{\eta}}$ is in $H^2_{\M}(A_{\eta},\Q(1))$.  The cycle $M_i \cap \Xi$ lies in $H^5_{\M}(A_{\eta},\Q(3))$ and has boundary 
	$$\partial(\Xi \cap M_i)=\sum_x \sum_j a^j_x (D^j_\eta \cap M_i)_x$$ 
	Let $g^i$ be the direct image $\pi_*(\Xi \cap M_i)$ in $H_{\M}^1(\eta,\Q(1))$. This is a function on the base $X$ with divisor 
	$$\div(g^i)=\sum a^i_x x$$
	Hence $(D_{\eta}^i,g^i)$ is a decomposable element with boundary $\sum_x a_x D^i_x$. Subtracting this from $\Xi$ for each $i$ gives us the cycle $\tilde{\Xi}$ with boundary 
	$$\partial(\tilde{\Xi})=\sum_x b_x S_x.$$
	\end{proof}

Applying this lemma to the cycles $\Xi_{\Delta}^c$ we get cycle $\tilde{\Xi}_{\Delta}^c$ whose boundary is supported only on the extra cycles. Since we know $\tilde{\Xi}_{\Delta}^c$ is indecomposable, $b_x$ is non-zero for $x$ lying on $H_{\Delta}^c \cap X$.

The Picard number of an Abelian surface in characteristic $0$ can be at most $4$ and over a modular curve it is $3$. If the Picard number is $4$ the Abelian surface is a product of isogenous $CM$ elliptic curves. This happens at $CM$ points $x$ on the modular curve. A choice of the generator of the orthogonal complement of the generic Neron-Severi is called the $CM$-cycle at $x$ and the divisor $S_x$ is essentially a multiple of the $CM$ cycle at that point.

\begin{rem}[A remark about components]. The number of components of $H_{\Delta}$ is either $6,10$ or $15$ depending on if $\Delta \equiv 5 \mod 8$, $\Delta\equiv 1 \mod 8$ or $\Delta \equiv 0 \mod 2$. Curiously, even in the case when the rational curve is a conic, all three possibilities occur. The configuration determined by a rational curve passing through $5$ points and tangent to the $6$th line is the case $\Delta=5$. There are $6$ components in this case and the number of possibilities of rational curves passing through $5$ points and tangent to the $6$th line is $6$. Similarly,  there are $15$ ways of choosing $2$ lines and $20$ ways of choosing $3$ lines and each of these choices corresponds to a component (its possible that more than one configuration corresponds to the same component). Something similar must hold for higher $d$ but the combinatorics is a little harder. 
		
		\end{rem}

\subsubsection{An example: Conics and Humbert's theorem.}
\label{z5construction}

In the case when $d$ is small one can do this construction quite explicitly. In this section we construct the motivic cycle in the case  $\Delta=5$  In this case one has the following beautiful theorem of Humbert. 

\begin{thm}[Humbert\cite{humb}] If $A_z$ is an Abelian surface with moduli point $z$ then $z$ lies on $H_5$ if and only if there is a {\bf smooth conic} $Q$ in $\CP^2_{A_z}$ passing through five of the fifteen points $q^{ij}$ and tangent to the sixth line. 
\end{thm}

As is well known, there is a conic passing through any $5$ points in $\CP^2$. Hence the cycles determining the element $Z_5$ will come from the conic passing through either the $5$ points $q^{ij}$ or four of the points $q^{ij}$  and the point of tangency. While in Theorem \ref{motiviccycle} we considered the curve passing through all but one of the $q^{ij}s$ in this argument we will use the curve which is no longer tangent to the sixth line. 

To make this explicit, we  use the conventions set up by Hashimoto-Murabayashi \cite{hamu}. Let $A$ be an Abelian surface. Let $\phi$ be the map to $\CP^3$ and assume $\phi(0)=[0,0,0,1]$. The Kummer surface $K_A$ can be though of as a double cover of the conic 
$$yz=x^2$$
in $\CP^2$. Three of the points of tangency can be chosen to be $a_4=0=[0,0,1]$, $a_5=1=[-1,1,1]$ and $a_6=\infty=[0,1,0]$ and the other three are $a_i=[-a_i,a_i^2,1]$, $i \in \{1,2,3\}$. In this set up the lines are
$$l^i:y+2a_ix+a_i^2z=0$$
In particular, the line $l^4:y=0$ and $l^6=z=0$. The hyperelliptic curve $H$ such that $A=J(H)$ is given by 
$$H:y^2=x(x-1)(x-a_1)(x-a_2)(x-a_3)$$
The points $q^{ij}$ can be computed easily --
$$q^{ij}=l^i \cap l^j =[-(a_i+a_j),2a_ia_j,2]$$
as long as $i \neq 6$ and 
$$q^{i6}=l^6 \cap l^i=[-1,2a_i,0].$$

\subsubsection{The equation of a conic passing through $5$ points}

We are interested in cases when the extra rational curve in $\CP^2$ is a conic. The equation of a conic passing through $5$ points is given as follows. Let $q_i=[x_i,y_i,z_i]$ $1\leq i \leq 5$ be the five points. Then the equation of the conic is given by $\det(A)=0$, where $A$ is the matrix 
$$A=\begin{pmatrix} x^2 & y^2 &z^2 & xy & xz & yz\\
x_1^2 & y_1^2 &z_1^2 & x_1y_1 & x_1z_1 & y_1z_1\\
x_2^2 & y_2^2 &z_2^2 & x_2y_2 & x_2z_2 & y_2z_2\\
x_3^2 & y_3^2 &z_3^2 & x_3y_3 & x_3z_4 & y_3z_3\\
x_4^2 & y_4^2 &z_4^2 & x_4y_4 & x_4z_4 & y_4z_4\\
x_5^2 & y_5^2 &z_5^2 & x_5y_5 & x_5z_5 & y_5z_5\\
\end{pmatrix}
$$

This gives us the following theorem of Humbert \cite{hamu}[Theorem 2.9].

\begin{thm} Let 
	$$Q:p_1x^2+p_2y^2+p_3z^2+p_4xy+p_5xz+p_6yz=0$$
	be the equation of the conic passing though $q^{12},q^{23},q^{34},q^{45}$ and $q^{51}$ Then 
	\begin{align*}
	& p_1=4a_1a_2a_3(a_1-a_2) \\ & p_2=a_1^2+a_3-a_1a_3^2(a_3^2-a_2^2+a_2-a_3)\\
	& p_3=a_1a_2a_3^2(a_1-a_2)\\
	& p_4=2((a_2a_3+a_3)a_1^2+(-a_2a_3^2+a_2-a_3)a_1-a_2a_3(a_2-a_3))\\
	& p_5=2a_1a_2a_3(a_1-a_2)(a_3+1)\\
	& p_6=(-a_2^2a_3-a_2^2+a_3^2+a_2)a_1^2+(-a_2^2a_3^2-a_3^2)a_1-a_2a_3(a_2-a_3)
	\end{align*}
	In general, the line $l^6:z=0$ will meet $Q$ at two points $s^6_1$ and $s^6_2$. The condition that this conic is tangent to $l^6$,  namely $s_1^6=s_2^6$, is  $p_4^2-4p_1 p_2=0$. If this holds the moduli point lies on a component of $H_5$.
	
\end{thm}

Note that these expressions are not symmetric in the $a_i$ and changing them around gives the conics corresponding to other choices of points.

\subsubsection{The motivic cycle.}

Let $C$ be the double cover of $Q$ in $K_A$. This is an irreducible rational curve  ramified at the points $q^{12}, q^{23}, q^{34}, q^{45}$ and $q^{51}$ as well as the points $s^6_1$ and $s^6_2$. In the $K3$ surface $\TK_A$ the strict transform $\BC$ remains an irreducible rational curve ramified at $2$ points. 

Let $p$ be the point $q^{45}=[\frac{-1}{2},0,1]$ with $p_1$ and $p_2$ being the points lying over it. Let $f_p$ be the function with divisor $\div(f)=p_1-p_2$. Let $g_p$ be the function on $E_p$ with divisor $\div(g_p)=p_2-p_1$. The cycle in $\TK_A$ is given by 
$$Z_5=(\BC,f_p)+(E_p,g_p)$$

\subsubsection{The function $f_p$.}

With our choice, we have assumed that the point $0$ on $A$ maps to $[0,0,0,1]$ and the $\CP^2_A$ is given by $[x,y,z,0]$. The projection map $\pi:\CP^3 \rightarrow \CP^2$ is given by $[x,y,z,w] \rightarrow [x,y,z,0]$.

We want to describe the function $f=f_p$ in terms of $x$ and $y$ in $\CP^2_A$, at least upto a  choice of sign and constant.

Recall that the conic in $\CP^2$ is given by $yz=x^2$. Let 
$$S(x,y,z)=\prod_{i=1}^{6} l^i(x,y,z)$$
be the equation of the (degenerate) sextic. The equation of the double cover of $\CP^2$ ramified at the six lines $l^i(x,y,z)$ is given by 
$$w^2=S(x,y,z)$$
The conic $Q$ meets the six lines at the points $q^{12},q^{23},q^{34},q^{45},q^{51}$ and the two points $s^6_1$ and $s^6_2$ so the double cover is ramified at those points. 

Let $Q(x,y,z)=0$ be the equation of $Q$. Then the double cover $C$ is given by 
$$C=\left\{ [x,y,z,w]| Q_1(x,y,z)=0, w^2=S(x,y,z) \right\} $$
and over a point $[x,y,z]$ one has the two points $[x,y,z,\pm\sqrt{S(x,y,z)}]$. 

$\TK_A$ is obtained by blowing up $K_A$ at the $15$ points $q^{ij}$ and the point $[0,0,0,1]$. $f$ is defined on the strict transform of $C$ by the difference of the two points lying over $q^{45}$. It suffices to define $f$ in the blow up of $C$ at $q^{45}$. Note that $Q$ meets the lines $l^4$ and $l^5$ transversally at $q^{45}$.

To compute the blow up we use local coordinates near the point $[-\frac{1}{2},0,1]$. The lines are given by $y=0$ and $2x+y+z=0$ in $\CP^2$ and since the point is not on $z=0$ we may assume the point lies in $\AB^2$ given by $z=1$. 

The equation of the double cover is given by 
$$w^2=S(x,y)=y(2x+y+1)\prod_{i\neq \{4,5\}} l^i(x,y,z)=y(2x+y+1)H(x,y,z)$$
The blow up on the double cover at the point $[-\frac{1}{2},0,1,0]$  has exceptional fibre $\CP^2$ with coordinates $[t,u,v]$. The equation of the blow up of $K_A$ is given by 
$$(x+\frac{1}{2})u=yt \hspace{1 in} yv=wu \hspace{1in} (x+\frac{1}{2})v=wt$$
$$ w^2=y(x+2y+1)H(x,y,1)$$
Locally near the point $q^{45}$, the equation of the conic $Q$ is of the form $y=(x+\frac{1}{2})G(x)$ where $G(-\frac{1}{2})\neq 0$ as $(-\frac{1}{2},0)$ is a smooth point of the conic. So in the blow up in the fibre $\AB^3_z \times \AB^2_{t}$, where $\AB_*$ means the affine plane where $*=1$ the equation of the  curve is 
$$y=(x+\frac{1}{2})G(x) \hspace {.5in} w^2=y(x+\frac{1}{2})H(x,y) \hspace{.5in} y=(x+\frac{1}{2})u \hspace{.5in} w=(x+\frac{1}{2})v$$
This gives 
$$(x+\frac{1}{2})^2v^2=(x+\frac{1}{2})^2G(x)(G(x)+2)H(x,y)$$
that is 
$$ v=\pm \sqrt{G(x)(G(x)+2)H(x,y)}$$
and we know that $G(-\frac{1}{2}) \neq \{0,-2\}$ since the lines meet the conic transversally. To compute $G(-\frac{1}{2})$ we can use implicit differentiation to get 
$$G(-\frac{1}{2})=\frac{p_1-p_5}{p_6}$$
which is a function of $a_1,a_2$ and $a_3$. 

Let 
$$v_0^{\pm}= \pm \sqrt{G(-\frac{1}{2})(G(-\frac{1}{2})+2)H(-\frac{1}{2},0)}$$
and let 
$$P_1=[-\frac{1}{2},0,1,0],[1,G(-\frac{1}{2}),v^+_0] \hspace{1in} P_2=[-\frac{1}{2},0,1,0],[1,G(-\frac{1}{2}),v^-_0]$$
The parameter $v$ gives a parametrization of the blown up curve since $x,y,w,t$ and $u$ can be expressed in terms of $v$. Over a point $[x,y,1]$ of $Q$ which is unramified there are two points $[x,y,1,\sqrt{S(x,y,1)}$ and $[x,y,z,-\sqrt{S(x,y,1)}]$. In the blow up, over the point $[x,y,1,w]$ there is one point, namely $[x,y,1,w],[1,u,v]$. In the region where the local expression for $y$ as a function of $x$ holds, one has 
$$u=\frac{y}{(x+\frac{1}{2})}=G(x) \hspace{1in} v=\frac{w}{(x+\frac{1}{2})}=\frac{\pm \sqrt{S(x,y,1)}}{(x+\frac{1}{2})}$$

  $w \neq 0$ and $x \neq -\frac{1}{2}$ as the point $[-\frac{1}{2},0,1]$ is a ramified point, namely the point $q^{45}$. 

Hence we can consider the function 
$$f_P(v)=c\frac{v-v_0^+}{v-v_0^-}$$
on the strict transform of the curve $C$. We choose the constant $c$ such that $f_P(s^6_1)=1$. This is a function with 
$$\div(f_P)=P_1-P_2$$
Hence the cycle $Z_5$ is given by 
$$Z_5=(\bar{C},f_P)+(E_P,g_P)$$
This is a element which is determined by $a_1,a_2$ and $a_3$ as long as the corresponding point does not lie on the component of $H_5$ determined by $p_4^2-4p_1p_2=0$.

\section{Algebraicity of values of higher Green's functions.}

In this section we discuss a conjecture of Gross-Kohnen-Zagier and its connection with the indecomposable cycles that we constructed in Theorem \ref{motiviccycle}.

\subsection{The conjecture of Gross-Kohnen-Zagier.}

Gross-Zagier \cite{grza} and later Gross-Kohnen-Zagier \cite{GKZ} conjectured that the values at CM points of certain Green's functions are logaratihms of algebraic numbers. To state it we need some background on higher Green's functions and weakly holomorphic modular forms. 

\subsubsection{Higher Green's functions.}

Let $X=\overline{\HH/\Gamma}$ be a modular curve. It is well known that this parameterises elliptic curves with certain additional structure. Gross and Zagier \cite{grza} considered certain  `higher' Green's functions on $X \times X$ defined as follows. Let $\bar{\Gamma}$ denote the image of $\Gamma$ in  $PSL_2(\ZZ)$. Let $k\geq 1$. The {\em Higher Green's function of weight $k$} for $X$ is defined by 
$$
G^X_{k}(z_1,z_2)=\frac{-2}{[\Gamma:\bar{\Gamma}]} \sum_{\gamma \in \Gamma} \QQ_{k} \left( 1+\frac{|z_1-\gamma z_2|^2}{2 \Im(z_1) \Im(\gamma z_2)}\right).
$$
Here $\QQ_s$ is the Legendre function of the second kind defined by the Laplace integral 
$$
\QQ_{s-1}(t):=\int_0^{\infty} \frac{du}{(t+\sqrt{t^2-1} \cosh(u))^{s}}, \hspace {1cm} t>1,s>1.
$$

\subsubsection{Harmonic weak Maass forms and weakly holomorphic forms.}

Let $\Gamma'= Mp_2(\ZZ)$ be the metaplectic group. If $L$ is an even lattice of type $(2,1)$  one has the Weil representation 
$$\rho_L:Mp_2(\ZZ) \longrightarrow \C[L'/L]$$
where $L'$ is the dual lattice. Let $M_{k,\rho_L}(\Gamma')$ and $M^!_{k,\rho_L}(\Gamma')$ denote the spaces of modular forms and weakly holomorphic modular forms of weight $k$ with respect to the Weil representation $\rho_L$  respectively. Let $H^!_{k,\rho_L}(\Gamma')$ denote the space of harmonic weak Maass forms of weight $k$ with respect to $\rho_L$. The definition of these spaces can be found in \cite{BEY}, Section 3.2, for instance. 

One has a differential operator 
$$\xi_k: H^!_{k,\rho_L}(\Gamma') \longrightarrow M_{2-k,\bar{\rho}_L}(\Gamma')$$
$$ \xi_k(f)(\tau)=y^{k-2} \overline {\left(-2i y^2\frac{\partial f}{\partial \bar{\tau}}\right)}$$
Let $H_{k,\rho_L}(\Gamma')$ denote the subspace of $H^!_{k,\rho_L}(\Gamma')$ which maps to the cusp forms $S_{2-k,\bar{\rho}_L}(\Gamma')$ under $\xi_k$. The space of weakly  holomorphic forms is  the kernel of $\xi_k|_{H_{k,\rho_L}(\Gamma')}$ and one has a sequence \cite{BEY} Section 3.2

$$ 0 \longrightarrow M^!_{k,\rho_L}(\Gamma') \longrightarrow H_{k,\rho_L}(\Gamma') \stackrel{\xi_k}{\longrightarrow} S_{2-k,\rho_{L}}(\Gamma') \longrightarrow 0$$
The space of harmonic weak Maass forms can also be identified with the space of `mock modular forms' and corresponding map to the space of modular forms is called the `shadow map'.

\subsubsection{The conjecture of Gross, Kohnen and  Zagier.}

Let $G_s(z_1,z_2)$ be the Green's function defined above for $\Gamma=SL_2(\ZZ)$. Let 
$$G_s^m(z_1,z_2)=G_s(z_1,z_2)|T_m=-2 \sum_{\substack{\gamma \in M_2(\ZZ) \\ \det(\gamma)=m}} Q_{s-1}\left(1+\frac{|z_1-\gamma z_2|^2}{2Im(z_1)Im(\gamma z_2)}\right)$$
be the translate of $G_s$ under the action of the Hecke correspondence $T_m$ on either of the two variables. Let $f=\sum_m c_f(m) q^m$ be a weakly holomorphic modular form of weight $-2j$ for $\Gamma$ and let 
$$G_{1+j,f}(z_1,z_2)=\sum_{m>0} c_f(-m) m^j G^m_{j+1}(z_1,z_2)$$
Note that this a finite sum. The Gross-Zagier conjecture the following. For a discriminant $d<0$ let $\OO_d$ denote the ring of integers of $\Q(\sqrt{d})$ and $H_d$ its Hilbert class field. 

\begin{conj}[Gross-Kohnen-Zagier \cite{GKZ}] Assume that $c_f(m) \in \ZZ$ for all $m<0$. Let $z_1$ be a $CM$ point of discriminant $d_1$ and $z_2$ a $CM$ point of discrimant $d_2$ such that $(z_1,z_2)$ does not lie  on $Z(f)=\sum c_f(-m) T(m)$, where $T(m)$ is the $m^{th}$-Hecke correspondence in $X \times X$. Then there is an $\alpha$ in $H_{d_1}\cdot H_{d_2}$ such that 
	$$(d_1d_2)^{j/2} G_{j+1,f}(z_1,z_2)=\frac{w_{d_1}w_{d_2}}{4}\cdot \log|\alpha|$$
	
\end{conj}
In fact the original conjecture of Gross-Zagier was stated in terms of relations among coefficients of modular forms of a certain weight. The link with weakly holomorphic modular forms comes via the work of Borcherds \cite{borc} where he shows that the space of relations among coefficients of  modular forms of weight $k$ is given by the space of weakly holomorphic modular forms of weight $2-k$ with principal part determined by the relation. A relation among coefficients of  modular forms of weight $k$ is a sequence $\{b_m\}$ such that for a modular form $g=\sum a_m q^m$ of weight $k$,
$$\sum b_m a_m=0$$
In this case from $f$ one gets the relation 
$$\sum c_f(-m)a_m=0$$
since for an eigenform form  $g$, $T(m)g=a_m g$.

The conjecture above states that the values of Green's functions at $CM$ points are, up to an explicit algebraic number, the logarithm of an algebraic number. Progress on this conjecture has been made by various people - Gross-Kohnen-Zagier themselves \cite{GKZ}, Zhang \cite{zhan}, Mellit \cite{mell}, Viazovska \cite{viaz}, Zhou \cite{zhou} and most recently Bruinier-Ehlen-Yang \cite{BEY}.

\subsection{Motivic cycles and algebraicity of values of  higher Green's functions.}

In this section we show that the existence of indecomposable motivic cycles implies algebraicity of values of Green's functions. In a different form this idea is due to Zhang \cite{zhan} and  independently  Mellit \cite{mell}.

\subsubsection{CM cycles.}

Let $E \rightarrow X$ be the universal elliptic curve over $X$ and $E_{\eta}$ the generic fibre. Let $A_{\eta}=E_{\eta} \times E_{\eta}$. The dimension of the Neron-Severi group of $A_{\eta}$  is $3$,  generated by $E\times \{0\},\{0\} \times E$ and the diagonal. The dimension of $H^{1,1}(A)$ is $4$. 

Let $\tau$ be an imaginary quadratic number. The elliptic curve  $E_\tau$ has $CM$ by $\ZZ[\sqrt{D}]$ where $D$ is the discriminant of $\tau$. Let   
$$Z_{\tau}=\Gamma_{\sqrt{D}} - E_{\tau} \times \{0\} +D \{0\} \times E_{\tau}$$
where $\Gamma_{\sqrt{D}}$ is the graph of multiplication by $\sqrt{D}$. Let 
$$S_{\tau}=c(Z_{\tau}-\sigma^*(Z_{\tau}))$$ 
where $\sigma(x,y)=(y,x)$ and $c$ is a real number chosen so that $(S_{\tau},S_{\tau})=-1$.
$S_{\tau}$ is called the $CM$ cycle at $\tau$. By construction $S_{\tau}$ is orthogonal to the generic Neron-Severi group. An alternative definition is a multiple of  $(\Gamma_{\sqrt{D}}-\Gamma_{-\sqrt{D}})$. In particular, there is an ellptic curve of degree $D$ on $E \times E$ at a $CM$ point corresponding to $\Q(\sqrt{-D})$.

For a non-cuspidal point $y$  on $X$, let $\omega_y$ be a holomorphic form on $E_y$ such that $\int_{E_y} \omega_y \bar{\omega}_y=1$. There is a constant $k$ {\em independent of $y$} such that if $y$ is a $CM$ point then the cycle class of the CM cycle $S_y$ is given by 
$$cl(S_y)=k (\omega_1 \bar{\omega}_2 + \bar{\omega}_1\omega_2):=\eta_y.$$
However, $\eta_y$ is defined for {\em all} non-cuspidal points $y$.

\subsubsection{Green's Currents and regulators.}

If $\tau$ is a CM point and $S_{\tau}$ the CM cycle, let $g_2(\tau)$ denote the Green's current for $S_{\tau}$ with the following properties \cite{zhan} Section 3.1,

\begin{itemize}
	
	\item $\frac{\partial \bar{\partial}}{\pi i} g_2(\tau) = \delta_{S_{\tau}}$
		\item The integral 
		$$\int g_2(\tau) \eta=0$$
		for any $\frac{\partial \bar{\partial}}{\pi i}$ closed form $\eta$ on $W(\C)$.
		\end{itemize}
Let $\eta_y$ be as above and 
	$$H_2(\tau,y)=\int_{E_y \times E_y} \eta_y g_2(\tau)$$
	Zhang \cite{zhan}, Proposition 3.4.1, shows 
	$$H_2(\tau,y)=\frac{1}{2}G^X_2(\tau,y)$$
	Namely, the Green's current of $S_{\tau}$ evaluated on the form $\eta_y$ is half the value of the Green's function of weight $2$ at the point $(\tau,y)$.

	\begin{rem} In fact one has CM cycles $S_{\tau}^k$ of codimension $k$ in the CM fibres of the compactification of the $2(k-1)^{th}$-self product of the universal elliptic curve. Zhang shows that weight $k$ Green's function are related to the Green's currents of these CM cycles. Precisely, there is a distinguished $(k-1,k-1)$ form $\eta_y^k$ in $E_y^{2k-2}$ such that 
		
		 $$H_k(\tau,y)=\int_{E_y^{2k-2}} \eta_y^k g_k(\tau)=            \frac{1}{2}G^X_k(\tau,y)$$
		
		\end{rem}
	\begin{rem}
		One can also define CM cyles in the case when the base is a Shimura curve \cite{bess} and one would expect a similar expression for the Green's current of these CM cycles in terms of Green's functions.

	\end{rem}

\subsubsection{Motivic cycles and algebraicity.}

The localization sequence in this case is as follows. Let $X$ be a modular curve, $W$, $A_{\eta}$ and $A_{\tau}$ be as before. 

$$\cdots \longrightarrow H^{3}_{\M}(A_{\eta},\Q(2)) \stackrel{\partial}{\longrightarrow} \bigoplus_{\tau \in X} H^{2}_{\M}(A_{\tau},\Q(1)) \longrightarrow H^{4}_{\M} (W,\Q(2)) \rightarrow \cdots$$
When $\tau$ is a CM point, the CM cycles are elements of $H^2_{\M}(A_{\tau},\Q(1))$. A relation of rational equivalence among the CM cycles is a certain element of the motivic cohomology group $H^3_{\M}(A_{\eta},\Q(2))$. The link between motivic cycles and algebraicity comes from the following proposition. 

\begin{prop}
	Let $X$ be a modular curve and $A_{\eta}$ the generic fibre of the universal abelian surface. Let $\Xi$ be a cycle in the group $H^3_{\M}(A_{\eta},\Q(2))$ with boundary $\partial(\Xi)=\sum_{\tau} a_{\tau} S_{\tau}$, where $S_{\tau}$ is the $CM$ cycle in the fibre over a $CM$ point $\tau$. Then 
	\begin{itemize}
		\item If $\eta_y$ is the $(1,1)$ form defined above, 
		$$\langle \reg(\Xi),\eta_y\rangle=\frac{1}{2}\sum a_{\tau} G_2^X(\tau,y)$$
		\item If $y$ is a CM point {\em outside the support} of $\partial(\Xi)$ then
	$$\sum a_{\tau} G_ 2^X(\tau,y)$$ 
	is the logarithm of an algebraic number. 
	\end{itemize}
\end{prop}

\begin{proof}
The real regulator of an  element of $H^{3}_{\M}(A_{\eta},\Q(2))$ is a current on  $(1,1)$-forms in the family. From \cite{soul} Chapter III, Theorem 1, for instance,  it can also be obtained as a linear combination of the Green's currents associated to the cycles appearing in the boundary. From Zhang's theorem, therefore, the regulator evaluated on the $(1,1)$ form $\eta_y$ is a linear combination of  higher Green's functions. Precisely if $\Xi$ is a motivic cycle in $H^3_{\M}(A_{\eta},\Q(2))$ with 
$$\partial(\Xi)=\sum a_{\tau} S_{\tau}$$
then, if $y$ is not in the support of the boundary, one has 
$$\langle \reg(\Xi),\eta_{y} \rangle = \frac{1}{2}\sum a_{\tau} G^X_2(\tau,y).$$

Recall  that if $y$ is a $CM$ point,  one has the CM cycle $S_y$ in $H^2_{\M}(A_y,\Q(1))$ with cycle class $\eta_{y}$. Further, if  $y$ is  outside the support of the boundary of the motivic cycle, the cycle $\Xi$ restricts to give a cycle $\Xi|_y=\Xi_y$ in the motivic cohomology group $H^3_{\M}(A_y,\Q(2))$. The regulator of the cycle $\Xi$ evaluated at $\eta_{y}$ can then be viewed as the regulator of the product of the cycles $\Xi_y$ and $S_{y}$ in the motivic cohomology groups of $A_y$: 
$$H^3_{\M}(A_{y},\Q(2)) \otimes H^2_{\M}(A_{y},\Q(1)) \longrightarrow H^5_{\M}(A_{y},\Q(3)).$$
Elements of the group $H^5_{\M}(A_y\Q(3))$ are of the form $\sum_z (z,a_z)$ where $z$ is a point defined over possibly an extension $k_z$ of $K$  and $a_z \in K_1(z)=k_z^*$. The regulator of such an element is $\log \prod |a_z|$.

Explicitly, if $\Xi_{y}=\sum (C_i,f_i)$ then $\Xi_{y} \cap S_{y}=\sum (f_i, C_i \cap S_{y})$. The regulator of this is 
$$\log \prod  \left|f_i(C_i \cap S_{y}) \right|$$
where $f_i$ evaluated at the divisor $C_i \cap S_y=\sum a_P P$ means $\prod_P f_i(P)^{a_P}$. 

Hence at a $CM$ point $y$ we have two expressions for the regulator -- one in terms of a special value  of a higher Green's function and the other in terms of the regulator of an element of  $H^5_{\M}(A_{y},\Q(3))$. This gives 
$$(\reg(\Xi_{y}),\eta_{y})=\frac{1}{2}\sum a_{\tau} G^X_2(\tau,y))=\log  \prod \left| f_i(C_i \cap S_{y})\right|$$
The element $\Xi_y$  and the cycle $S_y$ are defined over a number fields. Therefore  the number 
$$ \alpha = \prod f_i(C_i \cap S_y)$$
is the value of an algebraic function evaluated at algebraic points and is hence an algebraic number.  

\end{proof}

\begin{rem} While we have stated things for $G_2$ the proposition holds for $G_k$ for any $k \geq 1$. Namely existence of a motivic cycle implies algebraicity of values of a certain linear combination of Green's functions. Since we construct cycles in the case  $k=2$ we have stated the proposition only in that case. 
	
	\end{rem}

\begin{rem} A trivial example of this is when $k=1$. Here a relation of rational equivalence between $CM$ points is a function whose divisor is supported on the $CM$ points. If $f$ is a function with that divisor it  can be defined over a number field.  If $\div(f)=\sum a_{\tau} z_{\tau}$ then one has 
	$$\frac{1}{2}\sum a_{\tau} G^X_1(z_{\tau},y)=\log|f(y)|.$$
	If $y$ is a $CM$ point then clearly $f(y)$ is an algebraic number and hence the linear combination of the Green's functions of weight $1$ are logarithms of algebraic numbers  at $CM$ points.  
	\end{rem}

If $\Xi$ is any element of $H^3_{\M}(A_{\eta},\Q(2))$ then its regulator is a current on $(1,1)$ forms. If the base is a modular or Shimura curve, three of the $(1,1)$ forms are represented by algebraic cycles in the generic fibre itself. If $\eta_C$ is a form representing a cycle $C$ in the generic Neron-Severi then from the above argument, $\Xi_y \cap C_y$ is an element of $H^5_{\M}(A_y,\Q(3))$ and 
$$\langle \reg(\Xi),\eta_{C_y}\rangle = \reg (\Xi_y \cap C_y)= \log|F(y)|$$
for some function $F$ on the base $X$. This will have divisor at those points where $\Xi$ has a boundary. This is a linear combination of the form $\sum_{\tau} \frac{1}{2} G_1(\tau,y)$ for some $\tau$.  Hence the {\em matrix coefficients} of 
 the regulator are given by Green's functions of degrees $1$ and $2$.
 
 In general, the matrix coefficients of the regulator of an element  in $H^{2k-1}_{\M}(A_{\eta}^{2k-2},\Q(k))$ are given by Green's functions of degree $\leq k$. Note further that in the case $k=2$ for instance, for decomposable cycles the $G_2$ component is $0$. Hence only {\em indecomposable} elements give algebraicity results for higher Green's functions. In general for $k>2$ there are products of the form, for every finite extension $L$ of the base field $K(\eta)$,  
 $$D_L=\bigoplus_{0<l<k} H^{2l-1}_{\M}(A_L^{2k-2},\Q(l)) \otimes H^{2k-2l}_{\ M}(A_L^{2k-2},\Q(k-l)) \longrightarrow H^{2k-1}_{\M}(A_L^{2k-2},\Q(k))$$
Let 
$$H^{2k-1}_{\M}(A_{\eta}^{2k-2},\Q(k))_{decomp}=\bigoplus_{L/K(\eta)} Nm^L_{K(\eta)}(D_L)$$
and let $$H_{\M}(A^{2k-2}_{\eta},\Q(k))_{indecom}=H_{\M}(A^{2k-2}_{\eta},\Q(k))/H_{\M}(A^{2k-2}_{\eta},\Q(k))_{decomp}$$
its only cyles which are non-zero in this group that provide algebraicity relations for $G_k$. 

\begin{rem}
	We have stated things for $A$ but essentially the same statements hold for $\TK_A$ as the underlying motive is the same. So one can equally well work with $\TK_A$. We will do this in the example later.  
\end{rem}

\subsubsection{The cycles $\Xi_{\Delta}^c$.}

Consider the cycles $\Xi_{\Delta}^c$ restricted to the universal family over $X$. Alternately, the construction can be carried out on the universal Abelian surface $A=E \times E$ over $X$. $\Xi_{\Delta}^c$ will have a boundary at those points where $E \times E$ has multiplication by $\Q(\sqrt{\Delta})$. At such a point $\tau$  the rank of the Neron-Severi is $4$ and so by a theorem of Shioda and Mitani, the Abelian surface is $E_{\tau} \times E_{\tau}$, where $E_{\tau}$ is an elliptic curve with complex multiplication. In particular, one has a $CM$ cycle at this point. 

Using that we have the following theorem:

\begin{thm} Let $\Xi_{\Delta}^c=\sum (\TC,f_P)+(E_P,g_P)$ be the cycle constructed above with boundary on points $\tau$ on $H_{\Delta}^c \cap X$. Let $y$ be a $CM$ point lying outside the support of $H_{\Delta}^c \cap X$. Then 

	$$\langle \reg(  \Xi^c_{\Delta}),\eta_y  \rangle=\sum a_{\tau} G_2^X(\tau,y)=\log \left( \prod_{x \in C_P|_{y} \cap S_y} |f_P(x)|^{\ord_{C_P|_y \cap S_y}(x)} \prod_{x \in E_P|_{y} \cap S_y} |g_P(x)|^{\ord_{C_P|_y \cap S_y}(x)} \right)$$
	In particular, since the points $C_P \cap S_y$ and $E_P \cap S_y$  are algebraic and the functions $f_P$ and $g_P$  are defined over ${\bar \Q}$ this is the logarithm of an algebraic number.
	\label{algebraicity}
	\end{thm}

\begin{proof} Let  $H=H_{\Delta}^c$. From Theorem \ref{boundary}, the boundary of the element $\Xi^c_{\Delta}$ is of the form $a(C^1_H-C^2_{H})$ where $C^i$ are the new cycles in the fibre over $H$ and $a \neq 0$. Restricting this to $X$ shows that the boundary is a sum of the form $\sum a(C^1_{\tau}-C^2_{\tau})$ where $\tau$ runs through the points if intersection of $H$ with $X$. 
	
Using Lemma \ref{decomp},  we can modify $\Xi_{\Delta}^c$ by decomposable elements to get a cycle $\tilde{\Xi}_{\Delta}^c$ with boundary
$$\partial (\tilde{\Xi}^c_{\Delta})=\sum_{\tau} a_{\tau} S_{\tau}=0 \text { in } CH^2_{\hom}(W)$$
where $\tau$ runs through the points of  $H^c_{\Delta} \cap X$ and $S_{\tau}$ is the $CM$ cycle in the fibre  at $\tau$.

Let $\tilde{\Xi}^c_{\Delta}=\sum (C_i,f_i)$. From the remarks above, at a $CM$ point $y$ outside the support of the boundary of $\tilde{\Xi}_{\Delta}^c$ we have 
$$\langle \reg(  \tilde{\Xi}^c_{\Delta}),\eta_y  \rangle=\sum a_{\tau} G_2^X(\tau,y)
=\reg(\tilde{\Xi}_{\Delta}^c \cap S_y)$$
$$=\reg(\Xi_{\Delta}^c \cap S_y)=\log \left(\prod_i \prod_{x \in C_i|_{y} \cap S_y} |f_i(x)|^{\ord_{C_i|_y \cap S_y}(x)}\right)$$
which is the logarithm of an algebraic number.

The difference between $\Xi^c_{\Delta}$ and $\tilde{\Xi}^c_{\Delta}$ consists of decomposable cycles and for a decomposable cycle the regulator computed againsts $\eta_y$  is $0$, so for the purpose of computing the algebraic number, we can simply use the cycle $\Xi^c_{\Delta}$.  That is $\langle \reg(  \Xi^c_{\Delta}),\eta_y  \rangle=\langle \reg(  \tilde{\Xi}^c_{\Delta}),\eta_y  \rangle$ and we have 

$$\langle \reg(  \Xi^c_{\Delta}),\eta_y  \rangle=\sum a_{\tau} G_2^X(\tau,y)=\log  \left(\prod_{x \in C_P|_{y} \cap S_y} |f_P(x)|^{\ord_{C_P|_y \cap S_y}(x)} \prod_{x \in E_P|_{y} \cap S_y} |g_P(x)|^{\ord_{C_P|_y \cap S_y}(x)} \right) $$

\end{proof}

\subsubsection{Motivic cycles and the conjecture of Gross, Kohnen and Zagier.}

The link between the algebraicity results above and the conjecture of Gross-Kohnen and Zagier comes from the following. 

Recall that a {\em Heegner cycle} is a sum of $CM$ cycles of the same discriminant. 
In their seminal paper and the sequel with Kohnen, Gross and Zagier also conjectured (and proved in the case of points) that Heegner cycles give rise to coefficients of modular forms of a certain weight $k$. Specifically certain generating series whose coefficients are given by the height pairings of Heegner cycles are modular forms. Assuming that the height pairing is a perfect pairing it would imply that {\em a relation among coefficients of modular forms of a certain weight is equivalent to a relation of rational equivalence among Heegner cycles.}

On the other hand from what we have discussed above, the localization sequence implies that {\em  a relation of rational equivalence among the Heegner cycles is a motivic cycle in the generic fibre}. Further, as remarked earlier, a {\em relation among coefficients of modular forms} is a weakly holomorphic modular form. Hence this suggests there should be a link between motivic cycles on the one hand and weakly holomorphic modular forms on the other. 

Zemel \cite{zeme} shows that in some cases there are certain relations between Heegner cycles that can be realised as coming from modular forms - but these relations are {\em not} relations of {\em rational equivalence} - so cannot be used to obtain algebraicity relations as above. 

The link between these two comes via Green's functions. Precisely, let $A_{\eta}$ denote the generic fibre of the  universal Abelian surface over $X \times X$ which is  product of elliptic curves. In the fibre over a CM point $z$ one has a codimension $j$ cycle $S_z^j$ in $H^{2j}_{\M}(A_z^j,\Q(j)$. 

As before, let $\omega_z$ be a $1$ form on $E_z$ such that $\int_{E_z} \omega_z \bar{\omega}_z=1$. For $I \subset [1,2j]$ with $|I|=j$, let $\epsilon_I:[1,2j] \longrightarrow \{1,-1\}$ be defined by $\epsilon_I(k)=1 \Leftrightarrow k \in I$. Let $\omega_{z}^{-1}=\bar{\omega}_z$. Define 
$$\Omega^j_z=c_{2j}\sum_I \omega_{z_1}^{\epsilon_I(1)}\cdots \omega_{z_{2j}}^{\epsilon_I(2j)}$$
where $c_{2j}$ is such that at a $CM$ point $z$  this form is the class of the $CM$ cycle $S_z^j$.  Then we can conjecture	

\begin{conj} \label{conj1} Given a weakly holomorphic form $f$ of weight $-2j$  with integral principal part, then there is a motivic cycle $\Xi^j_f$ in $H^{2j+1}_{\M}(A_{\eta}^{j},\Q(j+1))$ where $\eta$ is the generic point of $X \times X$, such that its boundary supported on $Z(f)=\sum c_f(-m)T_m$.  Further, if $z_1$ is a $CM$ point  and $\Omega^j_z$ is the $(j,j)$ form in the fibres over $z_1 \times X$, one has 
$$\langle \reg(\Xi^j_f),\Omega^j_z \rangle = \frac{1}{2}G_{j+1,f}(z_1,z)$$
\end{conj}

\subsubsection{An algebraicity result.}

The cycles $\Xi_{\Delta}$ we constructed are defined over the Siegel modular threefold. When restricted to $X \times X$ they are cycles which have boundary supported on components of graphs of Hecke correspondences. 
$$H_{\Delta} \cap H_1 = \bigcup_{\substack{ x \in \ZZ_{\geq 0} \\ x \equiv \Delta \mod 4}}   T\left(\frac{\Delta-x^2}{4}\right)$$
If $z_1$ is a $CM$ point, then restricting this cycle to $z_1 \times X$ gives a cycle which has boundary supported on points of intersection of the Hecke correspondences with $z_1 \times X$. These are points of the form $(z_1,z_2)$ where $E_{z_1}$ is isogenous to $E_{z_2}$ by an isogeny of some degree. The boundary of cycle can be expressed in terms of cycles whose class is $\Omega_{z_1,z_2}$ 

Note that while the Picard number of the generic fibre over $z_1 \times X$ is $2$ it jumps by $2$ over points $(z_1,z)$ where $E_z$ has $CM$ by the same field as $E_{z_1}$. 

If $z$ is such a point which does not lie in the support of the boundary then the algebraicity argument Theorem \ref{algebraicity} implies that the value of the Green's function determined by the regulator of $\Xi_{\Delta}$ is the logarithm of an algebraic number. 

Combining this with the conjecture above would imply special cases of the Gross-Zagier conjecture when $z_1$ and $z_2$ have the same $CM$ field.

\subsubsection{Other algebraicity results.}

The results of the previous section are of a slightly different nature. Above we restrict our cycles to $z_1 \times X$, where the rank of the generic Neron-Severi of the universal family is $2$ and there are $CM$ cycles only at points $z$ where $E_z$ is isogenous to $E_{z_1}$. In the earlier section we have considered the cycles on the universal Abelian surface over a modular curve, where the rank of the generic Neron-Severi is $3$ and one has $CM$ cycles at every $CM$ point. The algebraicity results then hold for all $CM$ points.

To put these results in the context of the conjectures of Gross-Kohnen-Zagier and the work of Brunier-Ehlen-Yang we have to use some other work of \cite{BEY} on a one variable Green's function. In some case, for instance in the work of Mellit \cite{mell}, this can be used to prove special cases of the Gross-Kohnen-Zagier theorem.

Bruinier-Ehlen-Yang and others prove precise algebraicity results in a different manner.  They show that the one variable Green's function  $-2m^jG_{1+2j}(z(m,\mu),z)$ obtained by evaluating the two variable function at a Heegner divisor $z(m,\mu)$ is the lift of a weak Maass form $f_{m,\mu} \in H_{\frac{1}{2}-2j,\rho_L}$. They then have a computation of the value of this at a $CM$ point $z$ and show that when one considers a linear combination of such $f_{m,\mu}$ which is a weakly holomorphic modular form, the value is the logarithm of an algebraic number. 

However, the argument does not work for the cases of $G_{s,f}$ when $s$ is an even integer. The problem is that the sum of the corresponding Heegner cyles is $0$ as $z(m,\mu)=P_{m,r}+P_{m,-r}$ 
where $P_{m,r}$ is the Heegner divisor considered by Gross-Kohnen-Zagier. For even $s$ the Heegner cycles over $P_{m,r}$ and $P_{m-r}$ have opposite signs so they cycles cancel out. 

They resolve this issue by considering `twisted' Heegner cycles  one a new moduli space $X_{\Delta}$ which has two components. They then  show \cite{BEY} Theorem 7.11, that the Green's function evaluated at a twisted Heegner point in one varible is a Borcherds lift of a weakly holomorphic form of weight $\frac{1}{2}-j$. 

Precisely, they have the following theorem 

\begin{thm}[Bruinier-Ehlen-Yang, Theorem 7.11] Let $j\in \ZZ_{\geq 0}$. Let $d$ and $\Delta$ be fundamental discriminants with $(-1)^jd<0$ and $(-1)^j\Delta>0$. Let $m=|d|/4$. For $f$ in $H_{-2j}$
$$G_{j+1,f}(z_{\Delta}(m,h),(z,h))=-2^{j-1} \tilde{\Phi}_{\Delta}^j(z,h,Za^j_{d}(f))$$
where $\Delta$ is the `twisting' factor, $h$ determines the component of the modular curve, $z_{\Delta}(m,h)$ is a twisted Heegner divisor, $\tilde{\Phi}$ is a certain lift which depends on whether $j$ is even or odd. Finally 
$$Za^j_{d}:H_{-2j} \longrightarrow H_{\frac{1}{2}-j,\rho_L}$$
is the Zagier lift from weak Maass forms of weight $-2j$ to those of weight $\frac{1}{2}-j$ and representation $\rho_L$. 

\end{thm}

Brunier-Ehlen-Yang \cite{BEY}(7.10) in fact show that the Green's function $G_{1+j}$  evaluated at a twisted Heegner point can be realised as a lift of a weak Maass for $f_m$ of weight $\frac{1}{2}-j$ and repesentation $\rho_L$. 
$$\tilde{\Phi}_ {\Delta}^j(z,h,f_m)=4m^{\frac{j}{2}}G_{1+j}(z_{\Delta}(m,h),(z,h))$$

The construction of the motivic cycles in Theorem \ref{boundary} can be carried out over the universal family over  this new moduli space $X_{\Delta}$. Given $\Delta_1$ we can construct a motivic cycle $\tilde{\Xi}_{\Delta_1}^c$ in $H^3_{\M}(A_{\eta},\Q(2))$ where $\eta$ is the generic point of $X_{\Delta}$, such that it has a boundary at certain twisted Heegner points determined by $\Delta_1$. Our results in the previous section would imply that the regulator of this cycle restricted to other $CM$ points is the logarithm of an algebraic number. Hence the values of a certain linear combination of Green's functions are logarithms of algebraic numbers. Once again, the weak modular forms give relations between {\em Heegner cycles} while our cycles give relations between certain sums of $CM$ cycles - which need not be the full Heegner cycle - so to get Heegner cycles we should consider the sum $\Xi_{\Delta_1}=\sum_c \Xi_{\Delta_1}^c$.  It seems reasonable to conjecture

\begin{conj} 
	\label{conj2}
	There is a weakly holomorphic modular form $f_{\Delta_1}$ of weight $-\frac{1}{2}$ and representation $\rho_L$, where $L$ is the transcendental lattice of $H^2$ of the universal Abelian surface over $X_{\Delta}$ such that 
	
	$$\langle \reg(\Xi_{\Delta_1}),\Omega_{(z,h)} \rangle=\tilde{\Phi}_{\Delta}^1(z,h,f_{\Delta_1})=\sum a(z_{\Delta}(m,h)) 4\sqrt{m} G_2(z_{\Delta}(m,h),(z,h))$$
where $\Omega_{(z,h)}$ is the $(1,1)$ form corresponding to the symmetric square in the fibre over $(z,h)$ and is the class of the $CM$ cycle at a $CM$ point. 		
\end{conj}
Implicitly, in the opposite direction we can conjecture
\begin{conj}
	\label{conj3}
 Given a weakly holomorphic modular form $g$ in $M^!_{\frac{1}{2}-j,\rho_L}(\Gamma)$ there exists a motivic cycle $\Xi_g$ in $H^{2j+1}_{\M}(A_{\eta}^j,\Q(j+1))$ with boundary supported on the $CM$ points with discriminant determined by the principal part of $g$ such that 
 $$\langle \reg{\Xi_g},\Omega_{(z,h)}\rangle = \tilde{\Phi}_{\Delta}^j((z,h),g)$$
\end{conj}

One could also link this to the earlier conjecture by speculating that the motivic cycle constructed here should be obtained by a geometric construction - sometimes merely by restrition - from the motivic cycle conjectured to exist above.

\subsubsection{Evidence.}

When $j=0$ the conjecture suggests that given a weakly holomorphic modular form of weight $\frac{1}{2}$ there should be an element in the group $H^1_{\M}(\eta,\Q(1))$, where $\eta$ is the generic point of the modular curve. Elements of the group $H^1_{\M}(\eta,\Q(1))$ are simply functions and so the question is whether one can construct  `interesting' functions on modular curves from weakly holomorphic modular forms of weight $\frac{1}{2}$. This is precisely what is done in Borcherds \cite{borc}.  The functions have divisors on Heegner divisors and cusps. Another related construction is in Zagier \cite{zagi1984} - where he constructs functions with divisors on modular forms by restricting a function with divisor supported on  modular curves on a Humbert surface.

When $j=1$ we know there is a weakly holomorphic modular form $f_{\Delta_1}$ of weight $-\frac{1}{2}$ with principal part $\frac{1}{q^{\Delta_1}}$ \cite{hemu}. The Borcherds lift 
$$\tilde{\Phi}_{\Delta}^1(z,h,f_{\Delta_1})$$
will be a sum of Green's functions corresponding to certain Heegner cycles determined by $\Delta_1$. The singularities of this function are at those points where there is a cycle of Humbert invariant $\Delta_1$. The Green's functions appearing are the Green's currents of certain $CM$ cycles determined by $\Delta_1$. 

On the other hand the cycle 
$$Z_{\Delta_1}=\sum_c Z_{\Delta_1}^c$$ 
restricted to $X_{\Delta}$ (or alternately, constructed on the universal family over $X_{\Delta}$) has a boundary on the same set -- the points where there is a cycle of invariant $\Delta_1$. Further, up to a decomposable element, the boundary is a linear combination of $CM$ cycles determined by $\Delta_1$. Finally, the regulator of this cycle, evaluated on a suitable $(1,1)$ form can be expressed in terms of Green's functions of weight $2$. This suggests that the cycle associated to the form $f_{\Delta_1}$ is the cycles $Z_{\Delta_1}$.

\subsubsection{An example: Computation of algebraicity.}\label{computation}

In Section \ref{z5construction} we constructed the motivic cycle $Z_5$ in the group $H^3_{\M}(\TK_A,\Q(2))$. In this section we use it to get an algebraicity result as described above. The idea is to restrict this cycle to the universal family over a modular curve and then intersect with certain CM fibres which are amenable to computation. For the cycle to restrict to the generic fibre we have to assume that the modular curve does not lie on $H_5$.The CM cycles we consider are those coming from the intersection of the modular curve with a component of $H_4$, so we further assume the modular curve does not lie on $H_4$. 

If $A$ is the generic fibre of over a modular or Shimura curve $X$, the rank of the Neron-Severi group is $3$. One generator is given by the genus $2$ curve. Let $C_1$ and $C_2$ be cycles representing the other two classes and let $\Delta_1$ and $\Delta_2$ be their Humbert invariants. Then $X$ can be realised as a component of $H_{\Delta_1} \cap H_{\Delta_2}$. Conversely, Hashimoto \cite{hash} showed that given a modular curve one can make a choice of generators so that it can be realised as a component of the intersection of two Humbert surfaces. 

If $A=A_z$ is a special fiber where the Picard number increases to  $4$ then a theorem of Shioda and Mitani\cite{shmi} states that $A_z=E_{1,z} \times E_{2,z}$ where $E_{1,z}$ and $E_{2,z}$ are isogenous elliptic curves with CM by the same field $\Q(\sqrt{\Delta})$. Note that is is not necessarily as a principally polarised abelian variety - the polarization on $A$ could still be irreducible. In particular, though, there is a $CM$ cycle - a generator of the orthogonal complement of the generic Neron-Severi group. 

Further, the theorem implies that there is an elliptic curve on $A_z$ and so $A_z$ has to lie on $H_{n^2}$ for some $n$. An elliptic curve on an Abelian surface will map to a rational curve on $K_A$ which passes through three of the points $q^{ij}$. On $\CP^2_A$ this is a rational curve of degree $n-1$. The complete characterization is once again due to Humbert \cite{biwi}[Theorem 7.5]. In particular, in the case $n=2$  the converse also holds. There is a line passing through three of the points $q^{ij}$  if and only if  there is an elliptic curve of degree $2$ on $A$ or equivalently the moduli point of $A$ lies on a component of $H_4$. 

Let $E$ be the elliptic curve, $L_1$  the rational curve that it maps to in $K_A$ and $\tilde{L}_1$ be strict transform of $L_1$ in $\TK_A$. There will always be a `conjugate' elliptic curve, namely the kernel of the map from $A$ to $E$, and it will be of the same degree. Let $L_2$ denote its image in $K_A$ and $\TL_2$ its strict  transform in $\TK_A$. 

There is another characterization for  the moduli point of $A$ to lie on a component of $H_4$, once again due to Humbert  \cite{hamu}[Theorem 2.12]

\begin{prop} If $A=J(C)$ is an Abelian surface where $C$ is the curve
	$$C: Y^2=X(X-1)(X-a_1)(X-a_2)(X-a_3)$$
	Then there exists a conic $Q_2$ passing through the four points $q^{12},q^{23},q^{35}$ and $q^{51}$ and tangent to the lines $l^4:y=0$ and $l^6:z=0$ if 
	$$(a_1a_3-a_2)^2 \times \{4a_1a_2a_3((a_1+a_3)(a_2+1)-2(a_1a_3+a_2))^2-(a_2-1)^2(a_1-a_3)^2(a_1a_3+a_4)^2\}=0$$
	Further, if the first factor is $0$ the moduli point lies on a component of $H_4$ and if the second factor is $0$ the moduli point lies on a component of $H_8$.
	
\end{prop}

Hence if we assume that $a_1a_3-a_2=0$ there should be an exceptional line. A small calculation shows that  $q_{13}=[-(a_1+a_3),2a_1a_3,2], q_{25}=[-(a_2+1),2a_2,2]$ and $q_{46}=[-1,0,0]$ lie on the  line 
$$l: Y-a_1a_3Z=0$$
In the affine plane given by $Z=1$ this is simply the line $[X,a_1a_3,1]$. Let $H_4^{c'}$ be the component on which this line exists. 
The preimage of this line in $K_A$ is a double cover of $\CP^1$ which has nodes  at the ramification points. Hence is two rational curves meeting at those points. 
One of those curves is $L_1$ and the other is its `conjugate', $L_2$.

In general if the moduli point of $A_z$, $z$,  lies on a Humbert surface then it will determine a relation among the $a_1,a_2$ and $a_3$ - though perhaps not as simple a relation as those determined by $H_4$ and $H_5$. Hence if $z$ lies on a modular curve realised as a component of the intersection of two Humbert surfaces then the generic fibre will be the Jacobian of a genus $2$ curve with $2$ relations among the $a_1,a_2$ and $a_3$. Since we have assumed that neither of the two Humbert surfaces is $H_4$  the set of points where $a_2=a_1 a_3$ will be certain $CM$ points - namely the points of intersection of $H_4$ with the modular curve.  Let $S_z$ denote the $CM$ cycle in $A_z$. It will map to $(L_1-L_ 2)$ or $(L_2-L_1)$.

Since we have further  assumed that neither of the Humbert surfaces are $H_5$ the cycle $\Xi_5$ restricts to give an element of $H^3_{\M}(A_{\eta},\Q(2))$ where $A_{\eta}$ is the generic fibre of universal family over the modular curve. Let $Z_5$ be its image in $H^3(\TK_{A_{\eta}},\Q(2))$. Then by the argument above, the regulator of $Z_5 \cap (\TL_1-\TL_2)$ will be the logarithm of an algebraic number. On the other hand, it is the value of the Green's function 
$$G_2^5(z)=\sum a_d G_2(z_d,z)$$
where $z_d$ is a point on $H_5^c \cap X$ and the  $\partial(\Xi_5)=\sum a_d S_d$  where $S_d$ is the $CM$ cycle at $z_d$ and $z$ is the point corresponding to $A$. 
 We compute this value using the description of $Z_5$ in Section \ref{z5construction}.

\subsubsection{Intersections.} 

The double cover $L$  of this line is a smooth rational curve of degree $2$ ramified at the three points $q^{13}, q^{25}$ and $q^{46}$ and passing through $\phi(0)$. Its normalization in $\TK_A$ is an unramified double cover of $\CP^1$ and is hence two smooth conics. Let $\TL_1$ and $\TL_2$ be the two smooth conics and $L_1$ and $L_2$ their image in $K_A$. The strict transform of the  image of the $CM$ cycle is $\TL_1-\TL_2$ or the other way around. Let us assume the former.  Hence the value we are interested in is 
$$f(\TL_1 \cap \TC) / f(\TL_2 \cap \TC).$$
These points  lie over $l \cap Q_1$. Since the point at infinity on $l$ is $[1,0,0]$ and this does not lie on $Q_1$ we may assume that $z=1$.  Let 
$$c_1=[x_1,a_1 a_3,1] \text{ and } c_2=[x_2,a_1 a_3,1]$$
be the two points lying on $l\cap Q_1$. Then the points in $L_i \cap C$ lying over them are 
$$c_1^{\pm}=[x_1,a_1 a_3,1,\pm \sqrt{S(x_1,a_1 a_3,1)}] \text{ and } c_2^{\pm}=[x_2,a_1 a_3,1,\pm \sqrt{S(x_1,a_1 a_3,1)}]$$
One of $c_1^{\pm}$ and $c_2^{\pm}$ will lie on $L_1$ and the other two will lie on $L_2$. Since the normalization in an unramified double cover one of the components correspond to the positive square root of $S(x,a_1 a_3,1)$ and the other corresponds to the negative one. So we may assume that $c_1^+$ and $c_2^+$ lie on $L_1$ and $c_1^-$ and $c_2^-$ lie on $L_2$.  Since these points are {\em not} nodal points  there is precisely one point lying over each of them in $\TK_A$. 
 
To find the $c_i$ it amounts to finding $x_i$. To find $x_1$ and $x_2$ we have to solve the quadratic equation 
$$AX^2+BX+C=0$$
where $A=p_1=4(a_1a_3)^2(a_1-a_1a_3)$, 
$B=p_4a_1a_3+p_5=$ and  $C=p_2(a_1a_3)^2+p_3+p_6a_1a_3$, so the points $x_1$ and $x_2$ are 
$$x_i=\frac{-B\pm\sqrt{B^2-4AC}}{2A}.$$
and the two points of intersection of $l$ with $Q_1$ are 
$$c_1=[x_1,a_1 a_3,1] \text{ and } c_2=[x_2,a_1 a_3,1]$$
The points lying over them in $K_A$ are, as above,
$$c_i^{\pm}=(x_i,y,1,w_i^{\pm})=(x_i,a_3,1,\pm \sqrt{S(x_i,a_1a_3,1)}$$
From this we can find the corresponding value of 
$$v_i^{\pm}=\frac{w_i^{\pm}}{(x_i+\frac{1}{2})}$$ 
and 
$$f_P(v_i^{\pm})=c\frac{v_i^{\pm}-v_0^+}{v_i^{\pm}-v_0^-}$$ 
These values are functions of $a_1$ and $a_3$ and are well defined on the region of $H_4^{c'} \setminus H_5^c$. Finally the number we are interested in is 
$$\frac{f_P(v_1^+)f_P(v_2^+)}{f_P(v_1^-)f_P(v_2^-)}$$

The values of $a_1$ and $a_3$ are determined by the fact that the moduli point of $A$ lies on a modular curve

		\section{Complements.}

		\subsubsection{The `modular' complex.}
		
		Let $X_{\Delta}$ be the modular curve considered above, $A_{\eta}$ the generic fibre of the universal family of Abelian surfaces (in this case $E \times E$) and $A_x$ the special fibre over a point $x$. 
		
		Let $\Gamma'=Mp_2(\ZZ)$ as before. Let $L$ be the  transcendental lattice of $H^2(A_{\eta},\ZZ)$. Since the generic Picard number is $3$, this is an even integral lattice of type $(2,1)$. 
		
		Let 
		$$H^{2j+2}_{\M}(W^j,\Q(j+1))_{CM}=\Ker:\{ H^{2j+2}_{\M}(W^j,\Q(j+1))_{\hom} \longrightarrow H^{2j+2}_{\M}(A_{\eta}^j,\Q(j+1))\}$$
		This is the subgroup of null homologous codimension $(j+1)$ cycles generated by the $CM$ cycles.  A priori it is only contained in the kernel, but from the  Schoen \cite{scho} one knows that integrally the difference consists of cycles supported on cusps. A version of the Manin-Drifeld theorem shows that these cycles are torsion and so rationally they are $0$.   
		
		Conjectures \ref{conj2} and \ref{conj3} above link two different sequences - one coming from $K$-theory and the other coming from the theory of modular forms. On one hand, one has the localization sequence in motivic cohomology,
		$$0  \longrightarrow H^{2j+1}_{\M}(A^{j}_{\eta},\Q(j+1))_{ind} \stackrel{\partial}{\longrightarrow} \bigoplus_{x \in X_{\Delta}^{CM}} H^{2j}_{\M}(A^{j}_x,\Q(j)) \longrightarrow H^{2j+2}_{\M}(W^{j},\Q(j+1))_{CM} \longrightarrow 0 $$
		where $X_{\Delta}^{CM}$ denotes the set of $CM$ points on $X_{\Delta}$. On the other, one has a sequence in the theory of modular forms \cite{BEY} Section 3.2,
		$$ 0 \longrightarrow M^!_{\frac{1}{2}-j,\rho_L}(\Gamma') \longrightarrow H_{\frac{1}{2}-j,\rho_L}(\Gamma') \stackrel{\xi_{\frac{1}{2}-j}}{\longrightarrow} S_{\frac{3}{2}+j,\rho_L}(\Gamma') \longrightarrow 0$$
		The conjecture asserts that the two sequences are related. One has maps
		
	\begin{center}

		\begin{tikzpicture}[baseline= (a).base]
		\node[scale=.87] (a) at (0,0){
		\begin{tikzcd}[cells={nodes={minimum height=3em}}]
			0 \arrow[r] & H^{2j+1}_{\M}(A^{j}_{\eta},\Q(j+1))_{\text{ind}}  \arrow[r,"\partial"] \arrow[d,"\reg"] & \bigoplus_{x \in X_{\Delta}^{CM}} H^{2j}_{\M}(A^{j}_x,\Q(j+1))  \arrow[r] \arrow[d,"g_x"] &  H^{2j+2}_{\M}(W^j,\Q(j+1))_{CM}  \arrow[r]\arrow[d,"AJ"] & 0 \\
			& \GG^{0}_{j+1}(X) \arrow[r] &  \GG_{j+1}(X) \arrow[r] &  AJ^{j+1}(X)  \\
			0 \arrow[r] & M^!_{\frac{1}{2}-j,\rho_L}(\Gamma') \arrow[u,"\Phi_j"] \arrow[r] & H_{\frac{1}{2}-j,\rho_L}(\Gamma') \arrow[u,"\Phi_j"] \arrow[r,"\xi_{\frac{1}{2}-j}"] & S_{\frac{3}{2}+j,\bar{\rho}_L}(\Gamma') \arrow[r] \arrow[u] & 0
		\end{tikzcd}
	};
	
\end{tikzpicture}
\end{center}

		where $\reg$  is the map which takes $\Xi$ to $\langle \reg(\Xi),\eta^j_{y} \rangle$ and similarly $g_x$ is the Green's current evaluated on $\eta^j_y$.  $\GG_{j+1}(X)$ is the space of Green's functions of degree  $j+1$ and $\GG_{j+1}^0(X)$ is the subspace of Green's functions arising from the regulator. The conjecture asserts that the vertical maps have the same image. This is true for the  middle column  thanks to the work of Bruinier-Ehlen-Yang \cite{BEY} when $X=X_{\Delta}$. The conjecture is that the left maps have the same image. 
		
		Ramakrishnan \cite{rama} introduced the notion of a `modular' complex - a subcomplex of the complex that defines motivic cohomology (or higher Chow groups) of a Shimura variety. This is the subcomplex defined by Shimura subvarieties, boundary components, Hecke translates of all of these and the relations between them. He constructs elements in this modular complex. In fact all known examples of non-trivial elements of the motivic cohomology of Shimura varieties lie in this complex. 
		
		Analogously, in Kuga fibre varieties over a  Shimura variety, one can consider the subcomplex generated by modular subvarieties and special cycles over them, like the $CM$ cycles. The conjecture suggests that there should be a corresponding sequence of modular forms and that the regulators of these cycles can be computed in terms of Borcherds lifts of modular forms. In particular, the regulator of the cycles constructed by \cite{rama} should be given by Borcherds lifts.

		\subsubsection{Remarks on other $K$-groups.}
		
		The localization sequence above is between $K_1$ of the generic fibre and $K_0$ of the special fibres. One can also consider other parts of the localization sequence relating $K_{i+1}$ of the generic fibre  and $K_i$ of the special fibre.  One can then speculate if there are other sequences relating modular forms to special elements of the $K$-theory of modular varieties. 
		
		Let $X$ be a modular curve defined over a field $K$. From a   weakly holomorphic modular form $g$ of weight $\frac{1}{2}$, the original Borcherds lift $G=\Phi(g)$ gives a function on the modular curve with divisor on the Heegner points determined by the principal part of $g$. Given two such forms $f$ and $g$  one can consider the Steinberg symbol $\{F,G\}$ which gives an element of $K_2(K(X))$. The tame symbol of these elements is supported on the $CM$ points appearing in the divisors of $F$ and $G$. One can define a `higher CM point' to be a pair $(z_{\tau},a_{\tau})$ where $z_{\tau}$ is a CM point and $a_{\tau}$ is an element of $k_{\tau}^*=K_1(z_{\tau})$. The cycles above give relations between such points. It would be interesting to see whether the regulator of such cyles can be computed in terms of a Borcherds lift. Further, whether one can find a condition in terms of $f$ and $g$ for when $\{F,G\}$ lifts to $K_2(X)$. 
		
		\subsubsection{Remarks on higher dimensions.}
		
		In this paper we construct indecomposable cycles in the group $H^3_{\M}(A,\Q(2))$ which we expect to be related to weakly holomorphic forms of weight $-\frac{1}{2}$. In general, we expect there to be interesting cycles in $H^{2j+1}_{\M}(A^j,\Q(j+1))$.
		
		There is an external product defined at the level of cycles. This was insprired by the work of Raviolo \cite{ravi}. If $Z_1=\sum (C_i,f_i)$ is a cycle in $Z^{a}(X,1)$ and $Z_2=\sum (D_j,g_j)$ is a cycle in $Z^b(Y,1)$ then 
		$$Z_1 * Z_2=\sum(C_i \times D_j,f_ig_j)$$
		is a cycle in $Z^{a+b-1}(X \times Y,1)$ which gives a motivic cycle in the group $H^{2(a+b-1)-1}_{\M}(X \times Y, \Q(a+b-1))$. 
		It should be remarked that it is not clear if this external product is defined at the level of the higher Chow groups or motivic cohomology groups - since it is not clear if a boundary multiplied by a cycle gives a boundary. However, in many cases (as in our case) we have an explicit description at the level of cycles where this product can be carried out. 
		
		For instance, in our case we have the cyle $Z_{\Delta}$ in $H^3_{\M}(A,\Q(2))$. $Z_{\Delta}*Z_{\Delta}$ gives a cycle in $H^5_{\M}(A^2,\Q(3))$. So this suggests a way to construct elements in the group we are interested in. 
		
		However, we would like `indecomposable' elements - which lead to relations between Heegner cycles.  The notion of indecomposability becomes a little more complicated as in the case of $j=1$ there are only products from $H^1_{\M}(A,\Q(1)) \otimes H^{2}_{\M}(A,\Q(1))$. In general if $A$ is defined over a number field $K$, one has  products 
		$$\bigoplus_{L/K} H^{2k+1}_{\M}(A_L^j,\Q(k+1)) \otimes  H^{2(j-k)}_{\M}(A_L^j,\Q(j-k)) \longrightarrow H^{2j+1}_{\M}(A_L^j,\Q(j+1))\stackrel{Nm^L_K}{\longrightarrow} H^{2j+1}_{\M}(A^j,\Q(j+1)) $$
		and a cycle is indecomposable if it does not lie in the image of any of these maps for any $k<j$. In fact one appears to have a decomposition
		$$\bigoplus H^{2k+1}_{\M}(A^j,\Q(k+1))_{ind} \otimes  H^{2(j-1k)}_{\M}(A^j,\Q(j-k))$$
		On each of these summands the regulator is given in terms of the higher Green's function $G_{k+1}$.
		
		It appears that the cycles obtained by the product above are not indecomposable in this sense. In the example above they lie in the image of $H^3_{\M}(A^2,\Q(2)) \otimes H^{2}_{\M}(A^2,\Q(1))$.

	\bibliographystyle{alpha}
	\bibliography{AlgebraicCycles.bib}
	
		

\end{document}